\documentclass{amsart}
\usepackage[letterpaper,top=2.5cm,bottom=2cm,left=3cm,right=3cm,marginparwidth=3cm]{geometry}
\usepackage{setspace,amsmath,amsthm,amssymb,mathrsfs,ulem,enumitem,wasysym}
\usepackage[colorlinks=true, urlcolor=blue, linkcolor=red, citecolor=blue]{hyperref}
\newcommand{\D}{\,\mathrm{d}} \newcommand{\R}{\mathbb{R}} 
 
\newtheorem{lem}{Lemma} \newtheorem{prop}{Proposition} \newtheorem{thm}{Theorem} \newtheorem{rmk}{Remark} \newtheorem{claim}{Claim}[prop]
\title{Classification of Anti-symmetric Solutions to Fractional Lane-Emden System in Borderline Case}
\author[C.~Li]{Congming Li}
\address{School of Mathematical Sciences\\
	Shanghai Jiao Tong University\\
	Shanghai, 200240, China}  \email{congming.li@sjtu.edu.cn}

\author[Z.~Wang]{Zixuan Wang\textsuperscript{*}}
\thanks{\textsuperscript{*}Corresponding author.}
\address{School of Mathematical Sciences\\
	Shanghai Jiao Tong University\\
	Shanghai, 200240, China}
\email{math\underline{\ }wzx@sjtu.edu.cn}
\begin{document}
\large
\setcounter{section}{-1}	
\begin{abstract}
	We investigate the existence and non-existence of anti-symmetric solutions to the fractional Lane--Emden system in the half-space $\R^n_+$:
	$$\begin{cases}
		(-\Delta)^s u = v^p & \text{in } \R^n_+,\\
		(-\Delta)^s v = u^q & \text{in } \R^n_+,
		\end{cases}$$
	where $0<s<\frac{1}{2},p,q>0,2s+p=1$, with $u,v$ nonnegative in $\R^n_+$ and anti-symmetric with respect to $\partial\R^n_+$. Li and Zhuo [2023, Sci. China Math.] previously proved non-existence in $L_{2s}$, as well as non-existence in the extended space $L_{2s+1}$ when $1-2s<p,q<\frac{n+2s}{n-2s}$ and existence in $L_{2s+1}$ when $p+2s<1,q+2s<1$. In this paper, we fill the remaining borderline case $p=1-2s$ in $L_{2s+1}$ space. We establish a non-existence result for $1-2s=p\leq q\leq\frac{n+2s}{n-2s}$ and an existence result for $0<q<p=1-2s$. For the non-existence result, we combine a maximum principle for anti-symmetric functions with the method of moving planes to derive a Liouville--type theorem when $1-2s<q\leq\frac{n+2s}{n-2s}$. We then iterate the lower-bound estimates to obtain a contradiction through blow-up. To prove the existence of nontrivial solutions when $0<q<1-2s$, we construct a nonhomogeneous solution by a monotone iteration scheme in $L_{2s+1}$ as well as analyze a homogeneous profile. Surprisingly, the homogeneous profile serves as a universal lower bound for all solutions.
\end{abstract}
\maketitle
\section{Introduction}
	
	For a function $u\in C^{1,1}_{\mathrm{loc}}(\R^n)$, the fractional Laplacian $(-\Delta)^s,0<s<1$ is given by the singular integral
	$$(-\Delta)^s u(x)=c_{n,s}\,\text{P.V.}\int_{\R^n}\frac{u(x)-u(y)}{|x-y|^{n+2s}}\D y,$$
	where P.V. denotes the Cauchy principal value and $c_{n,s}$ is a normalization constant so that $(-\Delta)^s$ coincides with the Fourier multiplier $\mathcal{F}\big((-\Delta)^su\big)(\xi)=|\xi|^{2s}\mathcal{F}u(\xi)$ if $u$ is a Schwartz function. The fractional Laplacian arises naturally in many fields. In probability theory, \(-(-\Delta)^s\) is the infinitesimal generator of rotationally symmetric stable Lévy processes \cite{Bertoin1996}; related applications to Lévy flights and anomalous diffusion can be found in \cite{JespersenMetzlerFogedby1999,MetzlerKlafter2000}. In geophysical fluid dynamics, fractional dissipation appears prominently in the surface quasi-geostrophic equation. Early work of Constantin, Córdoba and Wu \cite{ConstantinCordobaWu2001} initiated the study of dissipative SQG equations, while the critical case was later resolved in \cite{KiselevNazarovVolberg2007,C_V2010}. In conformal geometry, Graham and Zworski \cite{GrahamZworski2003} introduced conformally covariant fractional operators through scattering theory, leading to further developments such as the fractional Yamabe problem; see \cite{ChangGonzalez2011,GonzalezQing2013}. The fractional Laplacian is also closely related to sharp functional inequalities. Lieb's celebrated work \cite{Lieb1983} determined the sharp constants and extremal functions in the  $L^2$ Hardy--Littlewood--Sobolev inequality	which is dual to the sharp fractional Sobolev inequality
	$$\|u\|_{L^{\frac{2n}{n-2s}}(\mathbb R^n)}^2\leq C_{n,s}\int_{\mathbb R^n}u(-\Delta)^s u\D x.$$
	The critical fractional equation $(-\Delta)^s u=u^{\frac{n+2s}{n-2s}}$ arises naturally as the Euler--Lagrange equation of the sharp fractional Sobolev inequality.  
	
	Classification of solutions to the (fractional) Lane--Emden equation has attracted considerable attention. We first recall the classical scalar equation
	\[
	-\Delta u=u^p \qquad \text{in } \mathbb R^n,~n\geq 3.
	\]
	Gidas, Ni, and Nirenberg \cite{gidas1979} established fundamental symmetry results for positive solutions. Gidas and Spruck \cite{Gidas_Spruck} proved that there is no nontrivial nonnegative entire solution when
	\[	1<p<\frac{n+2}{n-2}.	\]
	At the critical exponent
	$p=(n+2)/(n-2)$, positive entire solutions were completely classified by Caffarelli, Gidas, and Spruck \cite{Caffarelli1989}; see also the moving-plane proof of Chen and Li \cite{chen1991}. 
	
	The analogous scalar problem for the fractional Laplacian is
	\[	(-\Delta)^s u=u^p \qquad \text{in } \mathbb R^n,	\qquad 0<s<1.	\]
	The corresponding critical exponent is
	\[	p_s=\frac{n+2s}{n-2s}, \qquad n>2s.	\]
	Chen, Li, and Ou \cite{ChenLiOu} classified the positive solutions of the associated critical integral equation.
	Later, Chen, Li, and Li \cite{Chen2017} proved the non-existence of nontrivial nonnegative solutions in the subcritical range and radial symmetry in the critical case. On the other hand, in the supercritical range, Ao, Chan, Gonz\'alez, and Wei \cite{AoChanGonzalezWei2020} constructed smooth positive radially symmetric entire solutions.  
	
	As a natural generalization of the scalar theory, the Lane--Emden system
	\begin{equation}\label{frac_Lane_Emden}
		\begin{cases}
			(-\Delta)^s u=v^p,\\
			(-\Delta)^s v=u^q
		\end{cases}
		\text{in } \mathbb R^n,
	\end{equation}
	arises. In this paper, inspired by \cite{Zhuo2022,ZhuoLi2023}, we study anti-symmetric solutions to the fractional Lane--Emden system
	\begin{equation}\label{problem}
		\begin{cases}
			(-\Delta)^su=v^p,~\text{in}~\R^n_+,p>0;\\
			(-\Delta)^sv=u^q,~\text{in}~\R^n_+,q>0;\\
			u,v\in L_{2s+1}(\R^n)\cap C^{0,1}_{\mathrm{loc}}(\R^n_+)\cap C(\R^n),\\
			u(x',x_n)=-u(x',-x_n)\geq 0,\,v(x',x_n)=-v(x',-x_n)\geq 0.~\forall(x',x_n)\in\R^n_+.
		\end{cases}
	\end{equation}
	Here
	$$\R^n_+:=\big\{x=(x',x_n)\in\R^n:x_n>0 \big\},$$
	$$L_{2s+1}(\R^n):=\left\{u:\R^n\to\R\;\Big|\; \int_{\R^n} \frac{|u(x)|}{1+|x|^{n+2s+1}}\D x<\infty\right\},$$
	and the fractional Laplacian operator on an anti-symmetric function $u\in L_{2s+1}(\R^n)\cap C^{0,1}_{\mathrm{loc}}(\R^n_+)$ is defined by
	\begin{equation}\label{def_in_L2s+1}
		\begin{aligned}
			(-\Delta)^su(x):=&c_{n,s}\int_{\R^n_+}\Big(\frac{1}{|x-y|^{n+2s}}-\frac{1}{|x-y^*|^{n+2s}}\Big)\big(u(x)-u(y)\big)\D y\\
			&+2c_{n,s}u(x)\int_{\R^n_+}\frac{\D y}{|x-y^*|^{n+2s}},~x\in\R^n_+,
		\end{aligned}
	\end{equation}	
	where $y^*=(y',-y_n)$. This is due to an extra cancellation of the fractional kernel
	$$\frac{1}{|x-y|^{n+2s}}-\frac{1}{|x-y^*|^{n+2s}}=O(|y|^{-n-2s-1}),~x~\text{fixed},~y\to\infty.$$
	Furthermore, if $u\in C^\infty_\mathrm{c}(\R^n)$ is anti-symmetric, then
	$$(-\Delta)^su(x)=O(|x|^{-n-2s-1}),x\to\infty.$$
	Li and Zhuo \cite{ZhuoLi2023} studied system \eqref{problem} and obtained the following result:
	\renewcommand{\thethm}{A}  
	\begin{thm}(\cite[Theorem 1.2 and 1.3]{ZhuoLi2023})
		Given $0<p,q\leq\dfrac{n+2s}{n-2s}$, assume $u,v$ satisfy \eqref{problem}. 
		\begin{itemize}
			\item If $p+2s>1$ and $q+2s>1$, then $u=v=0$.
			\item If $p+2s<1$ and $q+2s<1$, then there exist nontrivial solutions  to \eqref{problem}.
		\end{itemize}
	\end{thm}
	\renewcommand{\thethm}{\arabic{thm}} \setcounter{thm}{0}
	
	When $p=1-2s$, the borderline exponent leads to a subtle problem in comparing the function itself and the reflected function when applying the method of moving planes. In this case, the powers of $|x|$ on both sides of a key estimate coincide and no direct bound on the minimizing point can be obtained. To overcome this difficulty, we use a weighted comparison between the two components to derive a localization estimate of the extremal point. As a result, we obtain a sharp dichotomy based on the sizes of $q$:
	\begin{thm}
		\begin{itemize}
			\item If $p=1-2s\leq q\leq\dfrac{n+2s}{n-2s}$ and $u,v$ satisfies \eqref{problem}, then $u=v=0$ is the only solution of \eqref{problem}.\label{non_exist}
			\item If $0<q<p=1-2s$, then there exist nontrivial solutions of \eqref{problem}.\label{exist}
		\end{itemize}
	\end{thm}

	In fact, we have the following extension of the results when $q\leq 1-2s\leq p$, not necessarily when $p+2s=1$:
	\begin{prop}\label{D=0_no_sol}
		If $p\geq1-2s\geq q>0$ and $(q+2s)p+2s=1$, then $u=v=0$ is the only solution of \eqref{problem}.
	\end{prop}
	\begin{prop}\label{D<0_exist_sol}
		If $p\geq1-2s\geq q>0$ and $(q+2s)p+2s<1$, then there exist nontrivial solutions of \eqref{problem}.
	\end{prop}
	We will prove Proposition~\ref{D=0_no_sol} and Proposition~\ref{D<0_exist_sol}. Then the non-existence when $p=q=1-2s$ in Theorem~\ref{non_exist} and the existence when $0<q<p=1-2s$ in Theorem~\ref{exist} follow as corollaries by setting $2s+p=1$.
		
	Historically, if $u$ and $v$ are positive functions in the whole space $\R^n$ with $s=1$, \eqref{problem} reduces to the classical Lane--Emden system
	\[
	\begin{cases}
	-\Delta u=v^p,\\
	-\Delta v=u^q
	\end{cases}
	\text{ in } \mathbb R^n.
	\]
	The Sobolev critical hyperbola
	\[
	\frac{1}{p+1}+\frac{1}{q+1}=\frac{n-2}{n}
	\]
	replaces the scalar critical exponent $\frac{n+2}{n-2}$. Serrin and Zou \cite{SerrinZou1998} showed that positive
	radial entire solutions exist when $(p,q)$ lies on or above the critical
	hyperbola. In the subcritical region, the Lane--Emden conjecture predicts the non-existence of positive entire solutions. A number of Liouville--type results were established by de Figueiredo and Felmer \cite{Figueiredo1994ALT}, Serrin and Zou \cite{SerrinZou1996}, and Pol\'a\v cik, Quittner, and Souplet \cite{Polacik2007}. The conjecture is known in dimensions $n\leq4$, with the four-dimensional case proved by Souplet \cite{Souplet2009}. Recently, OpenAI \cite{OpenAI2026} anounced a complete proof of the (H\'enon--)Lane--Emden conjecture by the localized potential-kernel virial and source-layer interval-pressure method.
	
	For the fractional counterpart of the classical Lane--Emden system \eqref{frac_Lane_Emden}, the natural analogue of the Sobolev critical hyperbola is
	\[\frac{1}{p+1}+\frac{1}{q+1}=\frac{n-2s}{n}.\]
	The integral form of the critical fractional Lane--Emden system also appears as the Euler--Lagrange equation of the sharp Hardy--Littlewood--Sobolev inequality for 
	\begin{equation}\label{HLS_inequality}
		\begin{aligned}
			\left|\iint f(x)|x-y|^{2s-n}g(y)\D x\D y\right|\leq&N_{\sigma,s,n}\|f\|_{L^\sigma(\R^n)}\|g\|_{L^\tau(\R^n)},\\
			1<\sigma,\tau<\infty,&\frac{1}{\sigma}+\frac{1}{\tau}=1+\frac{2s}{n},
		\end{aligned}
	\end{equation}
	see \cite[Theorem 2.3]{Lieb1983}. Quaas and Xia established Liouville--type theorems for \eqref{frac_Lane_Emden} in the whole space in a substantial range of exponents
	\cite{MR3538413}, using the Caffarelli--Silvestre extension together with moving-plane arguments. Inspired by \cite{OpenAI2026}, Mei and Wei \cite{MeiWei2026} presented a proof of the fractional Lane--Emden conjecture.
	
	A natural model of problems in the upper half space is to ask the functions to vanish in the lower half space, i.e. the Dirichlet problem. The boundary condition and the nonlocality of the fractional Laplacian create additional difficulties. For the scalar Dirichlet problem
	$$\begin{cases}
		(-\Delta)^s u=u^p & \text{in }\mathbb R^n_+,\\
		u=0 & \text{in }\mathbb R^n\setminus\mathbb R^n_+,
	\end{cases}$$
	Chen, Fang, and Yang \cite{Chen2015} proved Liouville--type theorems in the critical and subcritical cases without imposing growth restrictions at infinity. Quaas and Xia \cite{Quaas2015} obtained related non-existence results for fractional equations and systems in the half-space, while Luo and Zou \cite{LuoZou16} studied the corresponding fractional integral systems. In this paper, we instead focus on a different half-space formulation, anti-symmetric solutions, i.e. $u(x',-x_n)=-u(x',x_n)$.
		
	To prove Theorem~\ref{non_exist}, we apply the method of moving planes, which was originally developed by Alexandrov and then introduced to the PDE community by Serrin \cite{serrin1971} who considered an overdetermined boundary value problem and proved that the domain must be a ball and the solution is radially symmetric. The method is then developed by Gidas, Ni, and Nirenberg \cite{gidas1979,Gidas_Ni_Nirenberg} to unbounded domains, and by Troy \cite{troy1981} to systems of semilinear elliptic equations
	$$	\Delta u_i + f_i(u_1,\dots,u_k) = 0,\quad i=1,\dots,k,	$$
	in bounded domains. See also \cite{chen1991} for a simplification and \cite{ChenLiOu,Chen2017} for applications to fractional equations. 	
	
	The proof of the non-existence result when $1-2s<q\leq\dfrac{n+2s}{n-2s}$ requires several ingredients adapted specifically to anti-symmetric functions. We first establish lower-bound estimates for solutions of \eqref{problem}, i.e. 
	$$u(x)\geq c_Rx_n,~v(x)\geq c_Rx_n,~x\in B_R(0)\cap\R^n_+$$
	for any given $R$. The key point is a maximum principle adapted to functions that are simultaneously reflected with respect to the boundary of the half-space and to the moving plane, which plays an important role when we apply a direct method of moving planes after performing a Kelvin transform. This shows that any nontrivial solution must depend only on the vertical variable \(x_n\).
	
	This dimensional reduction transforms the original system in \(\mathbb R^n_+\) into a one-dimensional anti-symmetric fractional system on \(\mathbb R_+\). A second moving-plane argument, where we compare the values at $t\in(0,\lambda)$ and $2\lambda-t$, yields monotonicity of both components. Combining this monotonicity with the integral lower bounds produces an iterative sequence of polynomial growth estimates with blowing-up exponents. This establishes the desired Liouville--type theorem.

	The proof of the non-existence when $(q+2s)p+2s=1$, Proposition~\ref{D=0_no_sol}, is more direct. Instead of deriving Liouville--type theorems, we obtain lower--bound estimates for $u$ and $v$ in a fixed cone away from the tangential direction, which imply linear growth of $v$. This yields a blow-up in the Green representation.	
	
	In the proof of Proposition~\ref{D<0_exist_sol}, the existence argument is of a different nature. We begin by constructing suitable super-solutions and applying a monotone iteration scheme to obtain genuine nontrivial solutions belonging to the natural weighted space \(L_{2s+1}\). Besides, by computing the fractional Laplacian of anti-symmetric homogeneous functions of the form
	\[ 	u_\gamma(t)=|t|^{\gamma-1}t, \]
	we identify appropriate homogeneous profiles for the coupled system. The two classes of solutions differ due to the distinct asymptotic behavior at infinity. In fact, the homogeneous profile provides a lower bound for every positive solution; see Proposition~\ref{homogeneous_min}.
	
	We now describe some nontrivial techniques and observations in the proof. The first obstacle arises when applying the moving plane arguemnt. Due to the borderline exponent $p=1-2s$, the lower bound of the Kelvin transform of $v$ has the same decay rate as the term arising from the maximum principle. As a consequence, the previous
	argument does not directly yield a comparison between $\omega_\lambda$ and $\zeta_\lambda$, nor a localization estimate for the point where the negative minimum is attained. To overcome this difficulty, we compare $\omega_\lambda$ with $\epsilon\zeta_\lambda$ for some sufficiently small $\epsilon$ to ensure that $\epsilon|\zeta_\lambda|\leq|\omega_\lambda|$ at the extremal point and thereby obtain the required localization estimate.
	
	The second observation concerns the sign of the fractional Laplacian of the homogeneous profiles. For  $u_\gamma(t)=|t|^{\gamma-1}t$, the sign of $(-\Delta)^su_\gamma$ changes as the homogeneity exponent $\gamma$ crosses 1. The homogeneity exponents of the profiles in our construction are carefully checked to lie in $(2s,1)$.
	
	The third remarkable observation is on the homogeneity of the integrands in the iterations of the lower--bound estimates and the construction of the super-solutions. This homogeneity allows the relevant convolution with the fractional potential kernel to be computed explicitly, considerably simplifying the estimates compared with those in \cite{Zhuo2022,ZhuoLi2023}.
		
	The paper is organized as follows. Section~\ref{sec_notation_lem} introduces notation and establishes fundamental lemmas, including the maximum principle for anti-symmetric functions and the crucial lower--bound estimate. Section~\ref{sec_non_D<0} is devoted to the proof of non-existence part of Theorem~\ref{non_exist}, with the method of moving planes adapted to the anti-symmetric system, the reduction to one dimension, and the iterative growth argument leading to non-existence. Section~\ref{sec_non_D=0} contains the proof of Proposition~\ref{D=0_no_sol} by an iteration of lower--bound estimates. Section~\ref{sec_constrict_q<p} consists of the construction of both solutions of linear growth and homogeneous solutions when $(q+2s)p+2s<1$, which proves Proposition~\ref{D<0_exist_sol} with an additional comparison result.
\section{Notation and Preliminary Lemmas}\label{sec_notation_lem}

	Equation \eqref{problem} may be understood in the distributional sense. For this purpose, let $\phi\in C^\infty_\mathrm{c}(\R^n_+)$ and extend it to the whole space by $\phi(x^*)=-\phi(x),x\in\R^n_+$. Then
	$$(-\Delta)^s\phi(x)=O(|x|^{-2s-n-1}),\,x\to\infty.$$
	Therefore, we call a pair of anti-symmetric functions
	$$u\in L_{2s+1}(\R^n)\cap L^p_{\mathrm{loc}}(\R^n),v\in L_{2s+1}(\R^n)\cap L^q_{\mathrm{loc}}(\R^n)$$
	a weak solution of \eqref{problem} if they are continuous at any point on $\partial\R^n_+$ and satisfy
	\begin{equation*}
		\int_{\R^n_+}u\cdot(-\Delta)^s\phi=\int_{\R^n_+}v^p\phi;\;\int_{\R^n_+}v\cdot(-\Delta)^s\phi=\int_{\R^n_+}u^q\phi,~~~\forall\phi\in C^\infty_\mathrm{c}(\R^n_+).
	\end{equation*}
	However, standard regularity theory implies that $u,v\in C^{0,1}_{\mathrm{loc}}(\R^n_+)$. We provide the details below. In Theorem~\ref{non_exist}, Proposition~\ref{D=0_no_sol} and Proposition~\ref{D<0_exist_sol}, we always have $p<1$ or $q<1$. Hence, without loss of generality, we assume $p<1$ when deriving the regularity.
	
	For any closed ball $B\subset\R^n_+$ we consider
	$$u_1(x)=c\int_{B}\big(|x-y|^{2s-n}-|x-y^*|^{2s-n}\big)v(y)^p\D y,~u_2=u-u_1,$$
	where $c$ is chosen such that $(-\Delta)^su_2=0$ in $\mathcal{D}'(B)$, i.e.
	$$\int_{\R^n_+}u_2\cdot(-\Delta)^s\phi=0,~\forall\phi\in C^\infty_{\mathrm{c}}(\R^n_+).$$
	Notice that $u_1$ is anti-symmetric and the function
	$$x\mapsto\int_B|x-y^*|^{2s-n}v(y)^p\D y$$
	is smooth in $B$. Therefore the only regularity issue for $u_1$ is the singular integral
	$$\int_B|x-y|^{2s-n}v(y)^p\D y:$$
	\begin{lem}\label{Riezs_regular}
		Let $f$ be a nonnegative measurable function in $\R^n$ with compact support. Recall that the Riesz potential is defined by $I_{2s}f=|\cdot|^{2s-n}*f$. Then the following holds:
		\begin{equation*}
			I_{2s}f\in
			\begin{cases}
				L^{p^*}(\R^n),&~\text{if}~f\in L^r(\R^n),~1<r,2sr<n,r^*:=(nr)/(n-2sr);\\
				\dot C^\theta(\R^n),&~\text{if}~f\in L^r(\R^n),~n<2sr<\infty,\theta:=2s-n/r;\\
				C^{2s+\theta}(0),&~\text{if}~f\in C^\theta(0),~2s+\theta\notin\mathbb{Z}.
			\end{cases}
		\end{equation*}
	\end{lem}
	\begin{proof}
		The first statement is the classical Hardy-Littlewood-Sobolev inequality \eqref{HLS_inequality}. If $n<2sr<\infty$, since $f$ has compact support, H\"older's inequality implies that $I_{2s}f$ is well-defined pointwise. Then we estimate $I_{2s}f(x)-I_{2s}f(y)$ by H\"older's inequality:
		\begin{align*}
			|I_{2s}f(x)-I_{2s}f(y)|\leq&\left(\int_{\R^n}\big||x-z|^{2s-n}-|y-z|^{2s-n}\big|^{r'}\D z\right)^{\tfrac{1}{r'}}\|f\|_{L^r}\\
			=&\left(\int_{\R^n}\big||e_1-z|^{2s-n}-|z|^{2s-n}\big|^{r'}\D z\right)^{\tfrac{1}{r'}}\|f\|_{L^r}|x-y|^\theta.
		\end{align*}
		Here the integral converges, since $(2s-n)r'>-n>(2s-n-1)r'$. We refer to \cite{Li2021} for the proof of the last assertion.
	\end{proof}
	
	Since $(-\Delta)^su_2=0$ in $\mathcal{D}'(B)$, we claim that $u_2$ is smooth:
	\begin{lem}\label{s_harmonic_smooth}
		Let $w\in L_{2s+1}(\R^n)$ be anti-symmetric and $(-\Delta)^sw=0$ in the distributional sense in $B\subset\subset\R^n_+$. Then $w$ is smooth in $B$.
	\end{lem}
	\begin{proof}
		Choose $\eta\in C_c^\infty(\mathbb R^n)$ such that
		$$\eta(y^*)=\eta(y),\qquad\eta\equiv1$$
		in a neighborhood of $\overline B\cup\overline{B^*}$, and write
		$$w=w_1+w_2,\qquad w_1=\eta w,\qquad w_2=(1-\eta)w.$$
		Then $w_1,w_2$ are anti-symmetric and $w_1\in L_{2s}(\mathbb R^n)$.
		
		For $x\in B$, since $w_2(x)=0$, by anti-symmetry,
		$$(-\Delta)^s w_2(x)=-c_{n,s}\int_{\mathbb R^n_+}\left(\frac1{|x-y|^{n+2s}}-\frac1{|x-y^*|^{n+2s}}\right)(1-\eta(y))w(y)\,dy=:F(x).$$
		Since $\operatorname{supp}(1-\eta)$ is separated from $B$, there is no singularity in the above integral. Moreover, for $x$ in any compact subset of $B$ and $|y|$ sufficiently large,
		$$\left|D_x^\alpha\left(\frac1{|x-y|^{n+2s}}-\frac1{|x-y^*|^{n+2s}}\right)\right|\le\frac{C_{\alpha,B}}{|y|^{n+2s+1}}$$
		for every multi-index $\alpha$. Hence, by $w\in L_{2s+1}(\mathbb R^n)$, we may differentiate under the integral sign and obtain $F\in C^\infty(B).$
		
		Since $(-\Delta)^s w=0$ in $\mathcal{D}'(B)$ and $w_1\in L_{2s}(\R^n)$ is anti-symmetric,
		$$\langle w_1,(-\Delta)^s\phi\rangle=\int_{\R^n_+}(w-w_2)(-\Delta)^s(\phi-\phi^*)=-\int_{\R^n_+}(-\Delta)^sw_2\cdot(\phi-\phi^*)=-\int F\phi,$$
		where $\phi\in C^\infty_{\mathrm{c}}(B)$ and $\phi^*(x)=\phi(x^*)$. Thus $w_1\in L_{2s}(\R^n)$ satisfies $(-\Delta)^sw_1=-F$ in $\mathcal{D}'(B)$. The interior regularity for distributional solutions of the fractional Laplacian \cite{Li2021} implies $w_1\in C^\infty(B)$. Since $\eta\equiv1$ in a neighborhood of $\overline B$, we have
		$w=w_1$ in $B$. Therefore $w\in C^\infty(B)$.
	\end{proof}

	We claim that weak solutions are H\"older continuous: since $v^p\in L^{1/p}_{\mathrm{loc}}$, we have $u_1\in L^{(1/p)^*}$. But Lemma~\ref{s_harmonic_smooth} states that $u_2$ is smooth in $B$, hence $u\in L^{(1/p)^*}_{\mathrm{loc}}(B)$, where $B$ is an arbitrary ball in $\R^n_+$. Thus $u\in L^{(1/p)^*}_{\mathrm{loc}}(\R^n_+)$. Similarly $v\in L^{((1/p)^*/q)^*}_{\mathrm{loc}}$ if $2s(p^{-1})^*<qn$. Note that we derived that $v\in L^{\gamma_k}_{\mathrm{loc}}$ for $k=0,1$, where
	$$\gamma_0=1,~\gamma_{k+1}=\big(\big(\frac{\gamma_k}{p}\big)^*\frac{1}{q}\big)^*.$$
	We may iterate the arguments whenever $\dfrac{2s\gamma_k}{p}<n$ and $\big(\dfrac{\gamma_k}{p}\big)^*\dfrac{2s}{q}<n$. However, the process must stop after finitely many steps. Otherwise, we write
	$$\frac{1}{\gamma_{k+1}}=\frac{pq}{\gamma_k}-\frac{2s}{n}(q+1).$$
	If $pq<1$, the sequence converges to a negative limit; else, by a simple computation,
	$$1\leq pq<1+\dfrac{2s}{n}(q+1),$$
	it diverges to $-\infty$, a contradiction.
	
	Hence we must have $\dfrac{2s\gamma_k}{p}\geq n$ or $\big(\dfrac{\gamma_k}{p}\big)^*\dfrac{2s}{q}\geq n$ for some $k$. If $\dfrac{2s\gamma_k}{p}=n$ holds, we derive $u\in \mathop{\cap}\limits_{r<\infty}L^r_{\mathrm{loc}}(\R^n)$, hence $v$ is H\"older continuous. If$\big(\dfrac{\gamma_k}{p}\big)^*\dfrac{2s}{q}>n$, then the second assertion of Lemma~\ref{Riezs_regular} yields the H\"older continuity of $v$. The remaining cases are treated similarly. Thus either $u$ or $v$ is H\"older continuous. Then we apply the third assertion of Lemma~\ref{Riezs_regular} iteratively to derive the smoothness of $u,v$ in the upper half-space. Therefore, we do not distinguish between weak and classical solutions.
	
	The following lemma provides lower bounds for the solutions $u,v$:
	\begin{lem}\label{lower_bound}
		Let $u,v$ satisfy \eqref{problem}. Then
		$$u(x)\geq c\int_{\R^n_+}\left(|x-y|^{2s-n}-|x-y^*|^{2s-n}\right)v(y)^p\D y,\,\forall x\in\R^n_+.$$
		$$v(x)\geq c\int_{\R^n_+}\left(|x-y|^{2s-n}-|x-y^*|^{2s-n}\right)u(y)^q\D y,\,\forall x\in\R^n_+.$$
		In particular, if $u,v$ are not identically 0, they have a positive lower bound on each compact subset of $\R^n_+$. Furthermore, given any $R>0,~\min\{u(x),v(x)\}>c_Rx_n$ for $x\in B_R^+:=B_R(0)\cap\R^n_+$.
	\end{lem}
	\begin{proof}
		Let $\rho\in C^\infty_\mathrm{c}(B_2;[0,1])$ be radially symmetric and satisfy $\rho=1$ on $B_1$, $\rho^R(x):=\rho(x/R)$ for any $R>0$. Define
		$$u^R(x)=c\int_{\R^n_+}\left(|x-y|^{2s-n}-|x-y^*|^{2s-n}\right)\rho^R(y)v(y)^p\D y,\,x\in\R^n.$$
		Then $u^R$ is anti-symmetric and $u^R(x)=O(|x|^{2s-n-1})$ as $x\to\infty$. Hence
		$$u^R\in L_{2s+1}(\R^n)\cap C^{0,1}_{\mathrm{loc}}(\R^n_+)\cap C(\R^n).$$
		We choose $c=c_s'$ such that
		$$(-\Delta)^su^R=\rho^Rv^p~~\text{in}~~\R^n_+.$$
		The anti-symmetric function $w:=u-u^R$ satisfies
		$$(-\Delta)^sw=(1-\rho^R)v^p\geq0\text{  in  }\R^n_+;~~\liminf_{\R^n_+\ni x\to\infty}w(x)\geq0.$$
		We claim that $w\geq0$ in $\R^n_+$. Otherwise, there exists $x\in\R^n_+$ such that $w(x)=\inf\limits_{\R^n_+}w<0$. Then \eqref{def_in_L2s+1} yields a contradiction. As $R\to\infty$, the monotone convergence theorem implies
		$$u(x)\geq u^\infty(x):=c_s'\int_{\R^n_+}\left(|x-y|^{2s-n}-|x-y^*|^{2s-n}\right)v(y)^p\D y,\,\forall x\in\R^n_+.$$
		
		Let $K$ be any compact subset in $\R^n_+$. Then there exists a sequence $\{x_j\}\subset K$ such that
		$$\lim_{j\to\infty}u^\infty(x_j)=\inf_{x\in K}u^\infty(x).$$
		Passing to a subsequence, we may assume $x_j\to x_\infty\in K$. Then for every $y$,
		$$\left(|x_j-y|^{2s-n}-|x_j-y^*|^{2s-n}\right)v(y)^p\to\left(|x_\infty-y|^{2s-n}-|x_\infty-y^*|^{2s-n}\right)v(y)^p.$$
		By Fatou's lemma,
		$$\int_{\R^n_+}\left(|x_\infty-y|^{2s-n}-|x_\infty-y^*|^{2s-n}\right)v(y)^p\D y\leq\liminf_{j\to\infty}\int_{\R^n_+}\left(|x_j-y|^{2s-n}-|x_j-y^*|^{2s-n}\right)v(y)^p\D y,$$
		i.e.
		$$u^\infty(x_\infty)\leq\liminf_{j\to\infty}u^\infty(x_j)=\inf_{x\in K}u^\infty(x).$$
		Thus the infimum of $u^\infty$ over $K$ is attained, which must be positive if $v$ is not identically 0. Therefore $u$ has a positive lower bound on each compact subset of $\R^n_+$, so does $v$. For $x\in B_R^+$:
		\begin{align*}
			u(x)\geq&c_s'\int_{\R^n_+\cap B_R(2Re_n)}\left(|x-y|^{2s-n}-|x-y^*|^{2s-n}\right)v(y)^p\D y\\
			\geq&c_R\int_{\R^n_+\cap B_R(2Re_n)}\left(|x-y|^{2s-n}-|x-y^*|^{2s-n}\right)\D y\geq c_Rx_n,
		\end{align*}
		by the mean value theorem.
	\end{proof}
	
	To apply the method of moving planes, we introduce the following notation:
	$$\forall\lambda\in\R,\,T_\lambda:=\{x\in\R^n|x_1=\lambda\},\,\Gamma_\lambda=\{x\in\R^n_+|x_1<\lambda\}.$$
	Let $x^\lambda=(2\lambda-x_1,x_2,\cdots,x_n)$ be the reflection of the point $x=(x_1,x_2,\cdots,x_n)$ about the plane $T_\lambda$ and $x^*=(x',-x_n)$ be the reflection of $x$ about the boundary hyperplane $\R^{n-1}$ where $x'=(x_1,\cdots,x_{n-1})\in\R^{n-1}$.
	We will use the following maximum principle for the fractional Laplacian; its proof can be found in \cite{Zhuo2022}:
	\begin{lem}\label{maximum_principle}
		Let $w\in L_{2s+1}(\R^n)\cap C^{0,1}_\mathrm{loc}(\R^n_+)$ be such that
		$$w(y^*)=w(y^\lambda)=-w(y),\,\forall y\in\R^n.$$
		Assume that there exists $x\in\Gamma_\lambda$ such that $w(x)=\inf\limits_{\Gamma_\lambda}w<0$. Then
		$$(-\Delta)^sw(x)\leq c(n,s)w(x)(d^{-2s}+x_n^{-2s})$$
		where $d=\lambda-x_1$ and $c(n,s)$ is some positive constant depending only on $n,s$.
	\end{lem}
	
	Finally, we record a useful estimate.
	\begin{lem}\label{Li_observe}
		Let $p>0$. Then $b^p-a^p\simeq_p(a+b)^{p-1}(b-a),~0<a<b$.
	\end{lem}
	\begin{proof}
		If $b\geq 2a$, since $a+b\simeq b$ and $(1-2^{-p})b^p\leq b^p-a^p\leq b^p$, we derive
		$$b^p-a^p\simeq b^p\simeq(a+b)^{p-1}(b-a).$$
		Otherwise, if $a<b<2a$, then $a\simeq b\simeq a+b$. By the mean value theorem, there exists $\xi\in(a,b)$ such that $b^p-a^p=p\xi^{p-1}(b-a)$, where $\xi\simeq a+b$.
	\end{proof}
\section{Non-existence of Solutions When $p=1-2s<q\leq\dfrac{n+2s}{n-2s}$}\label{sec_non_D<0}
	
	To prove the non-existence when $p=1-2s<q\leq\dfrac{n+2s}{n-2s}$, we always assume $u,v$ are not identically 0, hence positive in $\R^n_+$ by Lemma~\ref{lower_bound}.
	\begin{prop}\label{depend_n}
		Let $1-2s=p<q\leq\dfrac{n+2s}{n-2s}$. Then $u,v$ depend only on the $x_n$ variable.
	\end{prop}

	We apply the Kelvin transform to $u$ and $v$:
	$$\phi(x)=\dfrac{1}{|x|^{n-2s}}u\big(\dfrac{x}{|x|^2}\big),~\psi(x)=\dfrac{1}{|x|^{n-2s}}v\big(\dfrac{x}{|x|^2}\big).$$
	Then $\phi(x)=o(|x|^{2s-n}),\,\psi(x)=o(|x|^{2s-n})$ as $x\to\infty$ and
	\begin{equation}
		(-\Delta)^s\phi(x)=\frac{\psi(x)^p}{|x|^\sigma},\,(-\Delta)^s\psi(x)=\frac{\phi(x)^q}{|x|^\tau},\,\forall x\in\R^n_+,
	\end{equation}
	where $\sigma=(n+2s)-(n-2s)p=2s\cdot(n+2-2s)>0,\,\tau=(n+2s)-(n-2s)q\geq0$.
	\begin{prop}\label{0symm}
		$\phi(x)=\phi(x^0),\,\psi(x)=\psi(x^0).$
	\end{prop}

	Proposition~\ref{depend_n} is an easy consequence of Proposition~\ref{0symm}. Indeed, assuming Proposition~\ref{0symm}, choose any two points $x^1,x^2\in\R^n_+$ at the same height. After a rigid transformation parallel to $\partial\R^n_+$ we may assume $x^1=(x_1,0,\cdots,0,x_n),x^2=(-x_1,0,\cdots,0,x_n)$. Proposition \ref{0symm} claims that $\phi(x^1)=\phi(x^2),\psi(x^1)=\psi(x^2)$, which immediately implies $u(x^1)=u(x^2),v(x^1)=v(x^2)$.
	
	To prove Proposition~\ref{0symm}, we introduce 
	$$\phi_\lambda(x):=\phi(x^\lambda),~\psi_\lambda(x):=\psi(x^\lambda)$$
	and
	$$\zeta_\lambda=\phi_\lambda-\phi,~\omega_\lambda=\psi_\lambda-\psi$$
	in $\Gamma_\lambda$ for $\lambda<0$. Then $\forall x\in\Gamma_\lambda$:
	\begin{equation}
		(-\Delta)^s\zeta_\lambda(x)=\frac{\psi(x^\lambda)^p}{|x^\lambda|^\sigma}-\frac{\psi(x)^p}{|x|^\sigma};
		(-\Delta)^s\omega_\lambda(x)=\frac{\phi(x^\lambda)^q}{|x^\lambda|^\tau}-\frac{\phi(x)^q}{|x|^\tau}.
	\end{equation}
	
	We first establish a uniform bound for a negative minimizer:
	\begin{lem}\label{min_point_est}
		Given $\lambda_0<0$. Then there exist $\epsilon>0,\,R<\infty$ such that the following holds.
		
		If $\lambda<\lambda_0$ and
		$$m:=\inf\limits_{x\in\Gamma_\lambda}\min\{\epsilon\zeta_\lambda(x),\omega_\lambda(x)\}<0,$$
		then there exists $\tilde x\in\Gamma_\lambda$ such that $m=\omega_\lambda(\tilde x)$. Moreover, $|\tilde x|\leq R$.
	\end{lem}
	\begin{proof}
		Let $\{x^k\}\subset\Gamma_\lambda$ be a minimizing sequence, i.e. $\lim\limits_{k\to\infty}\min\{\epsilon\zeta_\lambda(x^k),\omega_\lambda(x^k)\}=m$. Since
		\begin{equation}
			\liminf_{\Gamma_\lambda\ni x\to y}\omega_\lambda(x)\geq 0,\,\liminf_{\Gamma_\lambda\ni x\to y}\zeta_\lambda(x)\geq 0
		\end{equation}
		where $y\in\{\infty\}\cup\R^{n-1}\cup T_\lambda$ (this fact holds even if $y=0^\lambda$), the sequence $\{x^k\}$ is away from these boundary points. Hence there exists a subsequence converging to $\tilde x\in\Gamma_\lambda$, which satisfies $\min\{\epsilon\zeta_\lambda(\tilde x),\omega_\lambda(\tilde x)\}=m$ by continuity.
		
		Suppose, to the contrary, that $m=\epsilon\zeta_\lambda(\tilde x)\leq\omega_\lambda(\tilde x)$. By Lemma~\ref{maximum_principle}:
		\begin{equation}
			\frac{\psi(\tilde x^\lambda)^p-\psi(\tilde x)^p}{|\tilde x|^\sigma}\leq(-\Delta)^s\zeta_\lambda(\tilde x)\leq c(\tilde d^{-2s}+\tilde x_n^{-2s})\zeta_\lambda(\tilde x)<0,
		\end{equation}
		hence $\psi(\tilde x^\lambda)<\psi(\tilde x)$. By Lemma~\ref{Li_observe}, the left-hand side is comparable to $\big(\psi(\tilde x^\lambda)+\psi(\tilde x)\big)^{p-1}\dfrac{\omega_\lambda(\tilde x)}{|\tilde x|^\sigma}$. Since $p<1$ and $\dfrac{\omega_\lambda(\tilde x)}{\zeta_\lambda(\tilde x)}<\epsilon$, we apply Lemma~\ref{lower_bound} to derive:
		\begin{equation}
			c(\tilde d^{-2s}+\tilde x_n^{-2s})\leq\frac{C\epsilon}{|\tilde x|^\sigma}\big(\psi(\tilde x^\lambda)+\psi(\tilde x) \big)^{p-1}\leq\frac{C\epsilon}{|\tilde x|^\sigma}\psi(\tilde x)^{p-1}\leq\frac{c_0\epsilon}{|\tilde x|^\sigma}\left(\frac{\tilde x_n}{|\tilde x|^{n+2-2s}}\right)^{p-1}.
		\end{equation}
		Here $c_0$ is a constant depending on $\lambda_0$. By $(n+2-2s)(p-1)=-\sigma$, we derive $c\tilde x_n^{-2s}\leq c_0\epsilon\tilde x_n^{-2s}$, which is impossible if we choose $2c_0\epsilon=c$.
		
		Therefore $m=\omega_\lambda(\tilde x)<\epsilon\zeta_\lambda(\tilde x)$. We apply Lemma~\ref{maximum_principle} to $\omega_\lambda$:
		\begin{equation}
			\frac{\phi(\tilde x^\lambda)^q-\phi(\tilde x)^q}{|\tilde x|^\tau}\leq(-\Delta)^s\omega_\lambda(\tilde x)\leq c(\tilde d^{-2s}+\tilde x_n^{-2s})\omega_\lambda(\tilde x)<0.
		\end{equation}
		Hence $\phi(\tilde x^\lambda)<\phi(\tilde x)$, and by Lemma~\ref{Li_observe},
		$$\tilde d^{-2s}+\tilde x_n^{-2s}\leq C\dfrac{\phi(\tilde x)^{q-1}}{|\tilde x|^\tau}.$$
		We distinguish the following two cases:
		\begin{itemize}
			\item $q<1$. Similarly,
			$$\phi(\tilde x)^{q-1}\leq c_0\left(\dfrac{\tilde x_n}{|\tilde x|^{n+2-2s}}\right)^{q-1}\leq c_0\tilde x_n^{q-1}|\tilde x|^{(n+2-2s)(1-q)}.$$
			Hence
			$$\tilde x_n^{-2s}\leq\dfrac{Cc_0\tilde x_n^{q-1}}{|\tilde x|^{2(2s+q-1)}},$$
			which implies that $|\tilde x|^{2(2s+q-1)}\leq Cc_0\tilde x_n^{2s+q-1}$. Then $|\tilde x|\leq(Cc_0)^{\tfrac{1}{2s+q-1}}$.
			\item $q\geq1$. In this case $\phi(\tilde x)^{q-1}\leq C|\tilde x|^{(n-2s)(1-q)}$, hence
			$$\tilde x_n^{-2s}\leq\frac{C}{|\tilde x|^{\tau+(n-2s)(q-1)}}=\frac{Cc_0}{|\tilde x|^{4s}}.$$
			Thus $|\tilde x|\leq(Cc_0)^{\tfrac{1}{2s}}$.
		\end{itemize}
	\end{proof}

	\renewcommand{\proofname}{Proof of Proposition \ref{0symm}}
	\begin{proof}
		The above lemma implies
		\begin{equation}\label{zetaomega>0}
			\zeta_\lambda(x)>0,\omega_\lambda(x)>0,~\forall x\in\Gamma_\lambda
		\end{equation}
		if $\lambda<-R$. We move the plane $T_\lambda$ to the right (i.e. increase $\lambda$) as long as \eqref{zetaomega>0} holds. If the plane stops at some $\lambda_0$, i.e. 
		\begin{equation}
			\lambda_0:=\sup\big\{\mu<0|\eqref{zetaomega>0},\,\forall\lambda\leq\mu\big\}.
		\end{equation}
		We claim $\lambda_0=0$. Suppose, to the contrary, that $\lambda_0<0$. By continuity $\zeta_{\lambda_0}\geq0,\,\omega_{\lambda_0}\geq0$; the inequalities are in fact strict:
		$$\zeta_{\lambda_0}(x)>0,\,\omega_{\lambda_0}>0(x),~x\in\Gamma_{\lambda_0}.$$
		Otherwise, if there exists $x'\in\Gamma_{\lambda_0}$ such that $\zeta_{\lambda_0}(x')=0$, then
		\begin{equation}
			\begin{aligned}
				(-\Delta)^s\zeta_{\lambda_0}(x')=&c_{n,s}\int_{\R^n_+}\big(|x'-y|^{-(n+2s)}-|x'-y^*|^{-(n+2s)}\big) \big(\zeta_{\lambda_0}(x')-\zeta_{\lambda_0}(y)\big)\D y\\
				=&-c_{n,s}\int_{\Gamma_{\lambda_0}}\zeta_{\lambda_0}(y)g_{\lambda_0}(x',y)\D y\leq 0,
			\end{aligned}
		\end{equation}
		where
		$$g_{\lambda}(x,y)=|x-y|^{-n-2s}-|x-y^\lambda|^{-n-2s}+|x-y^{*\lambda}|^{-n-2s}-|x-y^*|^{-n-2s}>0,\,x,y\in\Gamma_\lambda.$$
		(See \cite{Li2010} for a proof of the fact $g_\lambda>0$). However, $$(-\Delta)^s\zeta_{\lambda_0}(x')\geq\dfrac{\psi(x'^{\lambda_0})^p}{|x'^{\lambda_0}|^\sigma}-\dfrac{\psi(x')^p}{|x'|^\sigma}>0,$$
		a contradiction. If $\omega_{\lambda_0}(x')=0$, a similar argument shows that $$0\leq\dfrac{\phi(x'^{\lambda_0})^q}{|x'^{\lambda_0}|^\tau}-\dfrac{\phi(x')^q}{|x'|^\tau}\leq(-\Delta)^s\omega_{\lambda_0}(x')\leq0,$$
		which is impossible if $\tau>0$. If $\tau=0$, we must have $\zeta_{\lambda_0}(x')=0$, which is false as shown.
		
		By the definition of $\lambda_0$, (in the following, we assume without loss of generality that $2\delta<|\lambda_0|$)
		\begin{equation}\label{futher_right_contradict}
			\forall\delta>0,\,\exists\lambda\in(\lambda_0,\lambda_0+\delta)\,\text{ s.t. }\inf\limits_{x\in\Gamma_\lambda}\min\{\epsilon\zeta_\lambda(x),\omega_\lambda(x)\}<0
		\end{equation}
		for $\epsilon>0$ as in Lemma~\ref{min_point_est}. Any minimizing sequence is bounded and away from the singularity $0^\lambda$, and hence has a subsequence converging to some limit $\bar x\in\Gamma_\lambda\setminus\{0^\lambda\}$. Hence the infimum is attained by $\omega_\lambda(\bar x)$.
		\begin{claim}\label{x1<<lambda0}
			$\exists\delta_0>0,l>0$ such that $\forall\lambda\in(\lambda_0,\lambda_0+\delta_0),~\bar x_1\leq\lambda_0-l$.
		\end{claim}
		\begin{claim}\label{xn>>h}
			$\exists h>0$ such that $\forall\lambda\in(\lambda_0,\lambda_0+\delta_0)$, there holds $\bar x_n\geq h$.
		\end{claim}

		Therefore we must have $\bar x\in D$, where
		$$D=\{x\in\Gamma_{\lambda_0-l}:|x|\leq R,~x_n\geq h\}$$
		is a compact subdomain of $\Gamma_{\lambda_0}$ independent of $\lambda$. However, by the (uniform) continuity of $\psi$, for $\lambda$ sufficiently close to $\lambda_0$ we have
		$$\sup_D|\omega_\lambda-\omega_{\lambda_0}|<\min_D\omega_{\lambda_0},$$
		which implies that $\omega_\lambda>0$ over $D$, contradicting \eqref{futher_right_contradict}. Hence such $\lambda_0$ does not exist.
	\end{proof}

	Here we prove Claim~\ref{x1<<lambda0}. Claim~\ref{xn>>h} is proved similarly, with $\bar x_n$ in place of $\bar d$.
	\renewcommand{\proofname}{Proof of Claim~\ref{x1<<lambda0}}
	\begin{proof}
		Suppose, to the contrary, that $\lambda_0-l<\bar x_1<\lambda$. Assuming $2l<\lambda_0$.
		\begin{itemize}
			\item  $q<1$. By the maximum principle (Lemma~\ref{maximum_principle}) and $|\bar x|\geq|\bar x_1|\geq\dfrac{\lambda_0}{2}$.
			\begin{equation}
			\bar d^{-2s}+\bar x_n^{-2s}\leq\frac{c}{|\bar x|^\tau}\phi(\bar x)^{q-1}\leq\frac{c\bar x_n^{q-1}}{|\bar x|^{2(2s+q-1)}}\leq c\bar x_n^{q-1}\min\{C,\bar x_n^{-2(2s+q-1)}\},
			\end{equation}
			hence
			$${\bar d}^{-2s}\leq c\bar x_n^{q-1}\min\{C,\bar x_n^{-2(2s+q-1)}\}-\bar x_n^{-2s}\leq C,$$
			which is impossible if $\delta_0,l$ are sufficiently small since $d\leq\delta_0+l$.
			\item $q\geq1$. Then
			$$\bar d^{-2s}+\bar x_n^{-2s}\leq\dfrac{c}{|\bar x|^\tau}\phi(\bar x)^{q-1}\leq\dfrac{c\max_{B_{2/|\lambda_0|}}|u|}{|\bar x|^{4s}}\leq C,$$
			which is impossible for sufficiently small $\delta_0$ and $l$.
		\end{itemize}
	\end{proof}

	By Proposition~\ref{depend_n}, we have the following simplification:
	\begin{equation}
		\begin{aligned}
			&(-\Delta)^su(x)\\
			=&c_{n,s}\int_{\R^n_+}\Big(\frac{1}{|x-y|^{n+2s}}-\frac{1}{|x-y^*|^{n+2s}}\Big)\big(u(x)-u(y)\big)\D y+2c_{n,s}u(x)\int_{\R^n_+}\frac{\D y}{|x-y^*|^{n+2s}}\\
			=&c_{n,s}\int_{\R^1_+}\big(u(x_n)-u(y_n)\big)\left\{\int_{\R^{n-1}}\Big(\frac{1}{|x-y|^{n+2s}}-\frac{1}{|x-y^*|^{n+2s}}\Big)\D y'\right\}\D y_n\\
			&+2c_{n,s}u(x_n)\int_{\R^1_+}\int_{\R^{n-1}}\frac{\D y'}{|x-y^*|^{n+2s}}\D y_n\\
			=&\bar c_{n,s}\int_{\R^1_+}\left(\frac{1}{|x_n-y_n|^{1+2s}}-\frac{1}{|x_n+y_n|^{1+2s}}\right)\big(u(x_n)-u(y_n)\big)\D y_n+2\bar c_{n,s}u(x_n)\int_{\R^1_+}\frac{\D y_n}{|x_n+y_n|^{1+2s}}\\
			=&c'(-\Delta)^su(x_n),
		\end{aligned}
	\end{equation}
	where, for $y_n\neq x_n$, the change of variables $x'-y'=(x_n-y_n)z$ gives
	$$\int_{\R^{n-1}}\frac{\D y'}{|x-y|^{n+2s}}=\frac{1}{|x_n-y_n|^{1+2s}}\int_{\R^{n-1}}\frac{\D z}{(|z|^2+1)^{s+n/2}}$$
	and similarly
	$$\int_{\R^{n-1}}\frac{\D y'}{|x-y^*|^{n+2s}}=\frac{1}{|x_n+y_n|^{1+2s}}\int_{\R^{n-1}}\frac{\D z}{(|z|^2+1)^{s+n/2}}.$$
	Write $u(x)=u(t),v(x)=v(t)$ where $t=x_n$. By a suitable scaling, system \eqref{problem} reduces to a one-dimensional problem
	\begin{equation}\label{1D_problem}
		\begin{cases}
			(-\Delta)^su=v^p;\\
			(-\Delta)^sv=u^q
		\end{cases}\text{  in  }\R_+,
	\end{equation}
	where $u,v\in L_{2s+1}(\R)\cap C^{0,1}_{\mathrm{loc}}(\R_+)\cap C(\R)$ are anti-symmetric.
	
	In the following, although the functions are only assumed to belong to $L_{2s+1}$, they are anti-symmetric with respect to some point, so that the fractional Laplacian is always well-defined. To be more specific, let
	$$\zeta\in L_{2s+1}(\R)\cap C(\R)~~\text{such that}~~\zeta(2\lambda-t)=-\zeta(t),\,\forall t\in\R,$$
	i.e. $\zeta$ is anti-symmetric with respect to $\lambda$. If $\zeta$ is $C^{0,1}$ at $t<\lambda$, we define the fractional Laplacian of $\zeta$ at $t$ (compared to \eqref{def_in_L2s+1}):
	\begin{equation}\label{1D_def}
		(-\Delta)^s\zeta(t)=c_{1,s}\int^\lambda_{-\infty}\left(\frac{1}{|t-r|^{1+2s}}-\frac{1}{|t+r-2\lambda|^{1+2s}}\right)\big(\zeta(t)-\zeta(r)\big)\D r+\frac{c_{1,s}\zeta(t)}{s(\lambda-t)^{2s}}.
	\end{equation}
	The maximum principle in one dimension is immediate:
	\begin{lem}\label{1D_maximum_pinciple}
		 If $t<\lambda$ attains the negative infimum of $\zeta$ over $(-\infty,\lambda)$, then
		$$(-\Delta)^s\zeta(t)\leq c\zeta(t)(\lambda-t)^{-2s}.$$ 
	\end{lem}
	\renewcommand{\proofname}{Proof}

	We show monotonicity of solutions for \eqref{1D_problem}. Instead of making a Kelvin transform as in the proof of Proposition~\ref{depend_n}, we deal with $u,v$ themselves directly. $\forall\lambda>0$, define $t^\lambda=2\lambda-t$, the reflection of $t$ with respect to $\lambda$; 
	$$u_\lambda(t)=u(t^\lambda),\,v_\lambda(t)=v(t^\lambda).$$
	Define the reflected differences by
	$$\tilde\zeta_\lambda=u_\lambda-u,\,\tilde\omega_\lambda=v_\lambda-v~~\text{in}~~\R.$$
	Then $\tilde\zeta_\lambda,\tilde\omega_\lambda\in L_{2s+1}(\R)$. Since $\tilde\zeta_\lambda$ and $\tilde\omega_\lambda$ are anti-symmetric with respect to $\lambda$, the fractional Laplacian of them is defined by \eqref{1D_def}. However, the equation of $u$ and $v$ is in the sense of definition \eqref{def_in_L2s+1}. We show that the two definitions coincide with each other and derive the equation of $\tilde\zeta_\lambda,\,\tilde\omega_\lambda$.
	\begin{lem}\label{tilde_equation}
		$\tilde\zeta_\lambda$ and $\tilde\omega_\lambda$ satisfy the following system in the sense of \eqref{1D_def},
		\begin{equation*}
			\begin{cases}
				(-\Delta)^s\tilde\zeta_\lambda=v_\lambda^p-v^p;\\
				(-\Delta)^s\tilde\omega_\lambda=u_\lambda^q-u^q
			\end{cases}~~\text{in}~~(0,\lambda).
		\end{equation*}
	\end{lem}
	\begin{proof}
		By definition, $\tilde\zeta$ is anti-symmetric with respect to $\lambda$. Fix $0<t<\lambda$ and define
		$$H(\tau):=\frac{u(2\lambda-t)-u(\tau)}{|2\lambda-t-\tau|^{1+2s}}-\frac{u(t)-u(\tau)}{|t-\tau|^{1+2s}}.$$
		Since
		$$H(\tau)=O(|\tau-t|^{-2s}),\,\tau\to t~~\text{and}~~H(\tau)=O(|2\lambda-t-\tau|^{-2s}),\,\tau\to2\lambda-t$$
		due to the regularity of $u$, and 
		$$|H(\tau)|\lesssim|\tau|^{-1-2s}+|u(\tau)||\tau|^{-2-2s},\,\tau\to\infty$$
		due to the cancellation of the kernel at infinity, $H$ is integrable on $\R$. By definition \eqref{def_in_L2s+1} with $n=1$ or \eqref{1D_def} with $\lambda=0$ after a reflection,
		$$(-\Delta)^su(y)=c_{1,s}\lim_{R\to\infty}\int_{-R}^{R}\frac{u(y)-u(\tau)}{|y-\tau|^{1+2s}}\D\tau.$$
		Applying this identity at $t$ and $2\lambda-t$, we obtain
		$$(-\Delta)^su(2\lambda-t)-(-\Delta)^su(t)=c_{1,s}\lim_{R\to\infty}\int_{-R}^{R}H(\tau)\D\tau=c_{1,s}\int_{\R}H(\tau)\D\tau.$$
		Since \(H\in L^1(\mathbb R)\), we may split the integral at \(\lambda\) and make a change of variable,
		$$\int_{\mathbb R}H(\tau)\D\tau=\int_{-\infty}^{\lambda}H(\tau)\D\tau+\int_{-\infty}^{\lambda}H(2\lambda-\tau)\D\tau=\int_{-\infty}^{\lambda}\bigl(H(\tau)+H(2\lambda-\tau)\bigr)\D\tau.$$
		A direct computation gives
		$$\begin{aligned}
		&H(\tau)+H(2\lambda-\tau)\\
		=&\frac{u(2\lambda-t)-u(t)+u(\tau)-u(2\lambda-\tau)}{|t-\tau|^{1+2s}}+\frac{u(2\lambda-t)-u(t)-u(\tau)+u(2\lambda-\tau)}{|t+\tau-2\lambda|^{1+2s}}\\
		=&\frac{\widetilde\zeta(t)-\widetilde\zeta(\tau)}{|t-\tau|^{1+2s}}+\frac{\widetilde\zeta(t)+\widetilde\zeta(\tau)}{|t+\tau-2\lambda|^{1+2s}}\\
		=&\left(\frac1{|t-\tau|^{1+2s}}-\frac1{|t+\tau-2\lambda|^{1+2s}}\right)\bigl(\widetilde\zeta(t)-\widetilde\zeta(\tau)\bigr)+\frac{2\widetilde\zeta(t)}{|t+\tau-2\lambda|^{1+2s}}.
		\end{aligned}$$
		Thus
		$$\begin{aligned}
		&(-\Delta)^su(2\lambda-t)-(-\Delta)^su(t)\\
		=&c_{1,s}\int_{-\infty}^{\lambda}\left(\frac1{|t-\tau|^{1+2s}}-\frac1{|t+\tau-2\lambda|^{1+2s}}\right)\bigl(\widetilde\zeta(t)-\widetilde\zeta(\tau)\bigr)\D\tau\\
		&+2c_{1,s}\widetilde\zeta(t)\int_{-\infty}^{\lambda}\frac{\D\tau}{|t+\tau-2\lambda|^{1+2s}}\\
		=&(-\Delta)^s\widetilde\zeta(t),
		\end{aligned}$$
		that is, $(-\Delta)^s\tilde\zeta_\lambda=v_\lambda^p-v^p$. The equation for $\tilde\omega$ follows similarly.
	\end{proof}
	
	\begin{prop}\label{increasing}
		$u,v$ are monotone increasing.
	\end{prop}
	\begin{proof}
		It suffices to show that $\tilde\zeta_\lambda,\,\tilde\omega_\lambda$ are always nonnegative. We use Lemma~\ref{tilde_equation} to establish the following claims.		
		\begin{claim}\label{1D_epsilonzeta_notinf}
			Let $\lambda<2\lambda_0$ for some fixed constant $\lambda_0>0$. Then there exists $\epsilon>0$, depending on $\lambda_0$, such that if
			$$\tilde m:=\inf\limits_{(0,\lambda)}\min\{\tilde\omega_\lambda,\epsilon\tilde\zeta_\lambda\}<0$$
			is attained by some $\tilde t\in(0,\lambda)$, then $\tilde m=\tilde\omega_\lambda(\tilde t)\leq\epsilon\tilde\zeta_\lambda(\tilde t)$.
		\end{claim}	
		Suppose, to the contrary, that $\tilde m=\epsilon\tilde\zeta_\lambda(\tilde t)<\tilde\omega_\lambda(\tilde t)$. Since $\tilde\zeta>0,\tilde\omega>0$ in $(-\infty,0)$, $\tilde m$ is the negative infimum of $\epsilon\tilde\zeta_\lambda$ over $(-\infty,\lambda)$. By Lemma~\ref{1D_maximum_pinciple}, we have
		\begin{equation}
			v_\lambda(\tilde t)^p-v(\tilde t)^p=(-\Delta)^s\tilde\zeta_\lambda(\tilde t)\leq c\tilde\zeta_\lambda(\tilde t)(\lambda-\tilde t)^{-2s}<0,
		\end{equation}
		hence $v_\lambda(\tilde t)<v(\tilde t),\,v_\lambda(\tilde t)^p-v(\tilde t)^p\geq pv_\lambda(\tilde t)^{p-1}\tilde\omega_\lambda(\tilde t)$. Therefore
		\begin{equation}\label{1D_epsilon_min}
			c(\lambda-\tilde t)^{-2s}\leq pv_\lambda(\tilde t)^{p-1}\frac{\tilde\omega_\lambda(\tilde t)}{\tilde\zeta_\lambda(\tilde t)}\leq c'\epsilon(2\lambda-\tilde t)^{-2s},
		\end{equation} 
		where we used $v(t)\geq ct$ for bounded $t$. Choosing $\epsilon$ sufficiently small, \eqref{1D_epsilon_min} is impossible. 
		\begin{claim}\label{d>l}
			Under the assumptions of Claim~\ref{1D_epsilonzeta_notinf}, there exists $l>0$ (independent of $\lambda$) such that $\lambda-\tilde t>l$.
		\end{claim}
		By Claim~\ref{1D_epsilonzeta_notinf}, $\tilde\omega_\lambda$ takes its negative infimum over $(-\infty,\lambda)$ at $\tilde t$. Lemma~\ref{1D_maximum_pinciple} yields
		\begin{equation}
			u_\lambda(\tilde t)^q-u(\tilde t)^q\leq(-\Delta)^s\tilde\omega_\lambda(\tilde t)\leq c\tilde\omega_\lambda(\tilde t)(\lambda-\tilde t)^{-2s}<0,
		\end{equation}
		which implies $u_\lambda(\tilde t)<u(\tilde t)$.
		\begin{itemize}
			\item $q<1$. By Lemma~\ref{lower_bound} $u(t)\geq ct$ for $0<t<4\lambda_0$, we have
			\begin{equation}			
			(\lambda-\tilde t)^{-2s}\leq cu_\lambda(\tilde t)^{q-1}\frac{\tilde\zeta_\lambda(\tilde t)}{\tilde\omega_\lambda(\tilde t)}\leq c(2\lambda-\tilde t)^{q-1}\leq c(\lambda-\tilde t)^{q-1}.
			\end{equation}
			\item $q\geq1$. Since $u$ is locally bounded, $(\lambda-\tilde t)^{-2s}\leq cu(\tilde t)^{q-1}\leq c\max\limits_{[0,2\lambda_0]}u^{q-1}$.
		\end{itemize}
		This proves the claim.
	
		Now we increase $\lambda$ from $0$ to $\infty$:\\
		\textbf{Step 1.} We first show that $\tilde\zeta_\lambda>0,\,\tilde\omega_\lambda>0$ for sufficiently small $\lambda$. Otherwise, assuming
		$$\tilde m:=\inf\limits_{(0,\lambda)}\min\{\tilde\omega_\lambda,\epsilon\tilde\zeta_\lambda\}<0,$$ there exists a sequence $\{t^k\}\subset(0,\lambda)$ such that $$\min\{\tilde\omega_\lambda(t^k),\epsilon\tilde\zeta_\lambda(t^k)\}\to\tilde m.$$
		Let $\tilde t\in[0,\lambda]$ be a limit point of $\{t^k\}$. Since $\tilde\omega_\lambda$ and $\tilde\zeta_\lambda$ are nonnegative on $(-\infty,0]\cup\{\lambda\}$, we must have $\tilde t\in(0,\lambda)$ and in fact $\tilde t$ is the minimum point on $(-\infty,\lambda]$. This is impossible if $\lambda<l$, by Claim~\ref{1D_epsilonzeta_notinf} and \ref{d>l}.\\		
		\textbf{Step 2.} We increase $\lambda$ as long as $\tilde\zeta_\lambda>0,\,\tilde\omega_\lambda>0$. If the process stops at some $\lambda_0<\infty$, the same argument as in the proof of Proposition~\ref{0symm} shows $\tilde\zeta_{\lambda_0}>0,\,\tilde\omega_{\lambda_0}>0$ in $(0,\lambda_0)$. By the definition of $\lambda_0,\,\forall k\geq2$, there exists $\lambda_k\in(\lambda_0,\lambda_0+l/k)$ such that
		$$\tilde m=\inf\limits_{(0,\lambda_k)}\min\{\tilde\omega_{\lambda_k},\epsilon\tilde\zeta_{\lambda_k}\}<0.$$
		By Claim~\ref{1D_epsilonzeta_notinf} and the fact that $\tilde\omega_{\lambda_k}$ and $\tilde\zeta_{\lambda_k}$ are nonnegative on $\{0,\lambda_k\}$, there exists $\bar t\in(0,\lambda_k)$ such that $\tilde m=\tilde\omega_{\lambda_k}(\bar t)$. Claim~\ref{d>l} gives $\bar t<\lambda_0-l/2$.
		However, since $\tilde\omega_{\lambda_0}>0$ on $(0,\lambda_0)$ and $\tilde\omega_{\lambda_0}(0)=v(2\lambda_0)>0$, by uniform continuity, $$\inf_{[0,\lambda_0-l/2]}\tilde\omega_{\lambda_k}\geq\inf_{[0,\lambda_0-l/2]}\tilde\omega_{\lambda_0}-\sup_{[0,\lambda_0-l/2]}|\tilde\omega_{\lambda_k}-\tilde\omega_{\lambda_0}|>0$$
		for $k$ sufficiently large. This yields a contradiction.
		
		Hence, we can increase $\lambda$ all the way to $+\infty$, which proves that $u,v$ are increasing.
	\end{proof}

	Now we complete the proof of the statement when $p=1-2s<q\leq\dfrac{n+2s}{n-2s}$ in Theorem~\ref{non_exist}. Lemma~\ref{lower_bound} implies:
	\begin{equation}\label{uv_grow}
		\begin{aligned}
			u(t)\geq&c_s'\int_0^\infty\big(|t-r|^{2s-1}-|t+r|^{2s-1}\big)v(r)^p\D r\\
			\geq&c_s'\int_t^\infty\big((r-t)^{2s-1}-(r+t)^{2s-1}\big)v(r)^p\D r\\
			\geq&ct\int_t^\infty(r+t)^{2s-2}v(r)^p\D r;\\
			v(t)\geq&ct\int_t^\infty(r+t)^{2s-2}u(r)^q\D r,
		\end{aligned}
	\end{equation}
	where we used the mean value theorem. By monotonicity, $u(t)\geq c$, $v(t)\geq c$ for $t\geq1$. Then \eqref{uv_grow} yields
	$$u(t)\geq ct\int_t^\infty(r+t)^{2s-2}\D r\geq ct^{2s}$$
	and similarly $v(t)\geq ct^{2s}$ for $t\geq1$. Denote
	$$\left(\begin{array}{c}
		\alpha_0\\
		\beta_0
	\end{array}\right)=
	\left(\begin{array}{c}
	0\\
	0
	\end{array}\right);~~
	\left(\begin{array}{c}
		\alpha_{k+1}\\
		\beta_{k+1}
	\end{array}\right)=
	\left(\begin{array}{c}
		2s\\
		2s
	\end{array}\right)+A
	\left(\begin{array}{c}
		\alpha_k\\
		\beta_k
	\end{array}\right),\text{  where  }A=
	\left(\begin{array}{cc}
		0 & p\\
		q & 0
	\end{array}\right).$$
	Then the above estimates read
	\begin{equation}\label{uv_kth_grow}
		u(t)\geq c_kt^{\alpha_k},\,v(t)\geq c_kt^{\beta_k},\,\forall t\geq 1
	\end{equation}
	 for $k=0,1$. Plugging \eqref{uv_kth_grow} into \eqref{uv_grow} inductively, we derive that \eqref{uv_kth_grow} holds for any $k\in\mathbb{N}$. We distinguish the following three cases.
	 \begin{itemize}
	 	\item $pq<1$. Then $A^k\to 0$, the sequence $(\alpha_k,\beta_k)$ converges to $\big(\dfrac{2s(p+1)}{1-pq},\dfrac{2s(q+1)}{1-pq}\big)$. The limiting exponents are also referred to as the scaling exponents of the Lane--Emden system, agreeing with those in \cite{MR3538413}. Hence $\forall\alpha\in\big(0,\dfrac{2s(p+1)}{1-pq}\big)$ we have $u(t)\geq ct^\alpha$ for $t>1$. Then \eqref{uv_grow} implies
	 	$$c\int_1^\infty(r+1)^{2s-2}r^{q\alpha}\D r\leq v(1).$$
	 	However, since $2s-2+\dfrac{2sq(p+1)}{1-pq}>-1$, the integral on the left-hand side diverges if $\alpha$ is sufficiently close to $\dfrac{2s(p+1)}{1-pq}$, which yields a contradiction.
	 	\item $pq=1$. Then $A^2=1,~\alpha_{2k}=2ks(p+1)\to\infty$. For some large $k$ the integral
	 	$$\int_1^\infty(r+1)^{2s-2}u(r)^q\D r\leq v(1)/c$$
	 	diverges as well, a contradiction.
	 	\item $pq>1$. $\alpha_{2k}=2s(1+p)\sum\limits_{j=0}^{k-1}(pq)^j\to\infty$, and the same argument applies.
	 \end{itemize}
 
\section{Non-existence of Solutions When $(q+2s)p+2s=1$}\label{sec_non_D=0}
	In this section, we prove Proposition~\ref{D=0_no_sol} via Lemma~\ref{lower_bound}. At the end of Section~\ref{sec_non_D<0}, we iterate the lower--bound estimate from Lemma~\ref{lower_bound} to derive
	$$v(x)\geq c(\alpha)x_n^\alpha,~x_n>1,~\forall\alpha\in\big(0,\dfrac{2s(q+1)}{1-pq}\big).$$
	This estimate remains valid when $(q+2s)p+2s=1$. However, since $\alpha<2s-2+\dfrac{2sp(q+1)}{1-pq}=-1$, all the corresponding integrals converge. To arrive at a blow-up, we need to pass through the limit as $k\to\infty$ to obtain the linear growth of $v$. Hence, it is crucial to keep track of the coefficients in the iteration.
	
	Assume that $(u,v)$ is a pair of positive solutions of \eqref{problem}. We estimate the functions $u$ and $v$ away from $\partial\R^n_+$. Let 
	$$\Gamma=\big\{\theta\in\mathbb{S}^{n-1}:\theta_n>\frac{1}{2}\}.$$
	To begin with, we provide a rough lower bound of $u$ and $v$ at infinity with direction in $\Gamma$.
	\begin{lem}\label{initial_lower_bound}
		There is a constant $c_0>0$ such that
		$$u(r\theta)\geq c_0r^{2s-n-1},~v(r\theta)\geq c_0r^{2s-n-1},~\forall r>1,\theta\in\Gamma.$$
	\end{lem}
	\begin{proof}
		Denote $x=r\theta$. By the mean value theorem, there is some $\xi$ between $|x-y|^2$ and $|x-y^*|^2$ such that
		\begin{equation}\label{kernel_at_infinity}
			\begin{aligned}
				\big(|x-y|^2\big)^{s-n/2}-\big(|x-y^*|^2\big)^{s-n/2}=&\big(s-\frac{n}{2}\big)\xi^{s-n/2-1}\big(|x-y|^2-|x-y^*|^2\big)\\
				\geq&(2n-4s)x_ny_n|x-y^*|^{2s-n-2}.
			\end{aligned}
		\end{equation}
		Let $\mathcal{B}=B_{1/3}(0',1/2)$. Then, for $r>1,\theta\in\Gamma,y\in\mathcal{B}$, \eqref{kernel_at_infinity} yields
		$$\big(|x-y|^2\big)^{s-n/2}-\big(|x-y^*|^2\big)^{s-n/2}\geq Cr^{2s-n-1}.$$
		Therefore
		$$u(r\theta)\geq Cr^{2s-n-1}\int_{\mathcal{B}}v(y)^p\D y\geq c_0r^{2s-n-1}.$$
		The lower bound of $v$ follows in the same manner.
	\end{proof}
	\begin{rmk}
		The corresponding initial estimate in Section~\ref{sec_non_D<0} is by the monotonicity, i.e. $u(x_n)\geq c,v(x_n)\geq c$ as $x_n>1$.
	\end{rmk}
	
	Then we start the iteration.
	\renewcommand{\proofname}{\textit{Proof of Proposition~\ref{D=0_no_sol}}}
	\begin{proof}
		Denote $x=r\theta,y=\rho\kappa$, where $\theta,\kappa\in\Gamma,~r\leq\rho$. By \eqref{kernel_at_infinity},
		$$\big(|x-y|^2\big)^{s-n/2}-\big(|x-y^*|^2\big)^{s-n/2}\geq(2n-4s)x_ny_n(r+\rho)^{2s-n-2}\geq C\rho^{2s-n-1}x_n\kappa_n.$$
		Then we derive		
		\begin{equation}\label{u_lower_bound}
			\begin{aligned}
				u(x)\geq&c\int_{\rho>r,\kappa\in\Gamma }\left(|x-y|^{2s-n}-|x-y^*|^{2s-n}\right)v(y)^p\D y\\
				\geq&Cx_n\int_r^\infty\int_\Gamma\rho^{2s-2}\kappa_n\cdot v(y)^p\D\kappa\D\rho.
			\end{aligned}
		\end{equation}
		Here $\kappa_n>\dfrac{1}{2}$ for $\kappa\in\Gamma$. Denote
		$$U(r)=\min_{\theta\in\Gamma}u(r\theta),~V(\rho)=\min_{\kappa\in\Gamma}v(\rho\kappa).$$
		Since $v(y)\geq V(\rho)$, \eqref{u_lower_bound} implies
		\begin{equation}\label{U_estimate}
			U(r)=\min_{\theta\in\Gamma}u(r\theta)\geq C_*r\int_r^\infty\rho^{2s-2}V(\rho)^p\D\rho.
		\end{equation}
		Here we emphasize that $C_*$ is a fixed constant. The same argument yields
		\begin{equation}\label{V_estimate}
			V(r)\geq C_*r\int_r^\infty\rho^{2s-2}U(\rho)^q\D\rho.
		\end{equation}
		Denote $A_0=B_0=c_0$ and $\alpha_0=\beta_0=2s-n-1$. By Lemma~\ref{initial_lower_bound}, we derive
		$$U(r)\geq C_*c_0^pr\int_r^\infty\rho^{2s-2}\rho^{p\beta_0}\D\rho=\frac{C_*B_0^p}{1-2s-p\beta_0}r^{2s+p\alpha_0}\geq A_1r^{\alpha_1};$$
		$$V(r)\geq C_*c_0^qr\int_r^\infty\rho^{2s-2}\rho^{q\alpha_0}\D\rho=\frac{C_*A_0^q}{1-2s-q\alpha_0}r^{2s+q\beta_0}\geq B_1r^{\beta_1},$$
		where $A_1:=C_*B_0^p,~B_1:=C_*A_0^q,~\alpha_1:=2s+p\beta_0,~\beta_1:=2s+q\alpha_0$. Here and in the following, we always assume that the integrals converge, i.e.
		$$p\beta_k<1-2s,~q\alpha_k<1-2s,$$
		where we define inductively
		$$\alpha_k:=2s+p\beta_{k-1},~\beta_k:=2s+q\alpha_{k-1},~\forall k\in\mathbb{N}^*.$$
		Moreover, denoting the coefficients inductively by
		$$A_k:=C_*B_{k-1}^p,~B_k:=C_*A_{k-1}^q.$$
		Then we apply \eqref{U_estimate} and \eqref{V_estimate} iteratively to derive
		\begin{equation}\label{UV_kth_estimate}
			U(r)\geq A_kr^{\alpha_k},~V(r)\geq B_kr^{\beta_k},~\forall k\in\mathbb{N},~r\geq 1.
		\end{equation}
		
		Since $pq=1-2s(p+1)<1$, arguments at the end of Section~\ref{sec_non_D<0} yields
		$$\alpha_k\to\frac{2s(p+1)}{1-pq}=1,~\beta_k\to\frac{2s(q+1)}{1-pq}=2s+q;~A_k\to C_*^{\tfrac{p+1}{1-pq}},~B_k\to C_*^{\tfrac{q+1}{1-pq}}.$$
		Then there exists a constant $c_*>0$ such that
		$$A_k\geq c_*,\,B_k\geq c_*,~\forall k\in\mathbb{N}.$$
		Now we can pass through the limit \eqref{UV_kth_estimate} to derive
		$$U(r)\geq c_*r,~V(r)\geq c_*r^{2s+q},~r\geq 1.$$
		Substituting in \eqref{U_estimate}, we derive
		$$U(r)\geq C_*c_*^pr\int_r^\infty\rho^{2s-2+(2s+q)p}\D\rho=\infty$$
		since $2s-2+(2s+q)p=-1$. This is a contradiction.
	\end{proof} 
	\renewcommand{\proofname}{\textit{Proof}}
	\begin{rmk}
		The one-dimensional reduction in Section~\ref{sec_non_D<0} is necessary. The behavior of $\alpha_k,\beta_k$ as $k\to\infty$ is vital in our proofs. If $pq<1$, the inductive relation of $(\alpha_k,\beta_k)$ is contractive. Then $(\alpha_k,\beta_k)$ converges to the scaling exponents $\big(\dfrac{2s(p+1)}{1-pq},\dfrac{2s(q+1)}{1-pq}\big)$ for any initial data $\alpha_0,\beta_0$. Here the rough bound $\alpha_0=\beta_0=2s-n-1$ suffices. However, if $pq\geq1$, the dynamical behavior of $(\alpha_k,\beta_k)$ depends heavily on the initial data. This is why we derive the constant lower bound of $u$ and $v$, i.e. $\alpha_0=\beta_0=0$, by the method of moving planes.
	\end{rmk}
\section{Existence of Solutions When $(q+2s)p+2s<1$}\label{sec_constrict_q<p}
	
	In this section, we prove Proposition~\ref{D<0_exist_sol}. It suffices to construct a solution for $n=1$. We first build solutions of linear growth at infinity as in \cite{ZhuoLi2023}. Then we carefully compute the fractional Laplacian of the homogeneous functions and find homogeneous solutions. At the end of this section, we show that the homogeneous profile is no larger than any other solution by an iteration scheme.	
	
	Inspired by the supersolution method in \cite{Zhuo2022,ZhuoLi2023}, we construct a bounded and nondecreasing sequence of lower solutions, which converges to a nontrivial solution.
	\begin{prop}
		There exists a solution of \eqref{1D_problem} such that $u(t)\geq t,\,\forall t>0$.
	\end{prop}
	\begin{proof}
		The construction consists of two steps. First, we construct an upper bound function $\big(\bar u(t),\bar v(t)\big)$, which is a supersolution of the related integral equation. Then, we iterate the Green representation to obtain a nondecreasing sequence of lower solutions which is bounded above by $\big(\bar u(t),\bar v(t)\big)$. The monotone convergence theorem then yields a positive solution of \eqref{1D_problem}, and hence \eqref{problem}, for $(q+2s)p+2s<1$.	
		
		\textbf{Step 1.} Define
		\begin{equation}
			\bar u(t)=K\cdot\big(t+t^\tfrac{2s(1+p)}{1-pq}\big),~\bar v(t)=M\cdot\big(t^{q+2s}+t^\tfrac{2s(1+q)}{1-pq}\big),
		\end{equation}
		for some large $K,M$ to be determined, such that $F(\bar u,\bar v)\leq(\bar u,\bar v)$, where
		\begin{equation}
			\begin{aligned}
				F_1(u,v)(t)=&t+c_s'\int_0^\infty\big(|t-r|^{2s-1}-|t+r|^{2s-1}\big)v(r)^p\D r;\\
				F_2(u,v)(t)=&c_s'\int_0^\infty\big(|t-r|^{2s-1}-|t+r|^{2s-1}\big)u(r)^q\D r,
			\end{aligned}
		\end{equation}
		where $c_s'$ is as in the proof of Lemma~\ref{lower_bound}. For any functions $u_1,u_2,v_1,v_2$ defined on $\R_+$, we write $(u_1,v_1)\leq(u_2,v_2)$ to mean that $u_1(t)\leq u_2(t)$ and $v_1(t)\leq v_2(t)$ for any $t>0$. Before further computation, we would like to point out
		$$\frac{2s(1+q)}{1-pq}<q+2s,~\frac{2s(1+p)}{1-pq}=p\cdot\frac{2s(1+q)}{1-pq}+2s<(q+2s)p+2s<1.$$
		
		By the inequality $(a+b)^p\leq C(p)\cdot(a^p+b^p)$ for $a,b>0,\,0<p<1$, the change of variables $r=tx$ yields:
		\begin{equation}
			\begin{aligned}
				&M^{-p}\int_0^\infty\big(|t-r|^{2s-1}-|t+r|^{2s-1}\big)\bar v(r)^p\D r\\
				\leq&C(p)\int_0^\infty\big(|t-r|^{2s-1}-|t+r|^{2s-1}\big)r^{(q+2s)p}\D r\\
				&+C(p)\int_0^\infty\big(|t-r|^{2s-1}-|t+r|^{2s-1}\big)r^\tfrac{2sp(1+q)}{1-pq}\D r\\
				=&C(p)\bar c_{s,(q+2s)p}t^{(q+2s)p+2s}+C(p)\bar c_{s,2sp(1+q)/(1-pq)}t^{\tfrac{2s(1+p)}{1-pq}},
			\end{aligned}
		\end{equation}
		where
		$$\bar c_{s,\gamma}:=\int_0^\infty\big(|1-x|^{2s-1}-|1+x|^{2s-1}\big)x^\gamma\D x,\,0<\gamma<1-2s.$$
		Since
		$$t^{(q+2s)p+2s}\leq t+t^{\tfrac{2s(1+p)}{1-pq}},$$
		by taking $M$ sufficiently large, we derive
		\begin{equation}
			F_1(\bar u,\bar v)(t)\leq t+CM^p\cdot\big(t+t^{\tfrac{2s(1+p)}{1-pq}}\big)\leq CM^p\cdot\big(t+t^{\tfrac{2s(1+p)}{1-pq}}\big)=\frac{CM^p}{K}\bar u(t).
		\end{equation}
		Similarly,
		\begin{equation}
			F_2(\bar u,\bar v)(t)\leq CK^q\cdot\big(t^{q+2s}+t^\tfrac{2s(1+q)}{1-pq}\big)=\frac{CK^q}{M}\bar v(t).
		\end{equation}
		Since $pq<1$, we can choose $K=CM^p$ and $M$ such that $C^{1+q}M^{pq}\leq M$. Then $F(\bar u,\bar v)\leq(\bar u,\bar v)$.
		
		\textbf{Step 2.} We define a sequence $\{(u_k,v_k)\}$ by
		$$(u_0,v_0)=0;~~(u_k,v_k)=F(u_{k-1},v_{k-1}).$$
		Then the sequence is nondecreasing in $k$ and bounded above by $(\bar u,\bar v)$, and therefore converges to a positive pair $(\hat u,\hat v)\leq(\bar u,\bar v)\in L_{2s+1}\times L_{2s+1}$ pointwise. By the monotone convergence theorem,
		\begin{equation}
			\left\{\begin{aligned}
				\hat u(t)=&t+c_s'\int_0^\infty\big(|t-r|^{2s-1}-|t+r|^{2s-1}\big)\hat v(r)^p\D r;\\
				\hat v(t)=&c_s'\int_0^\infty\big(|t-r|^{2s-1}-|t+r|^{2s-1}\big)\hat u(r)^q\D r.
			\end{aligned}\right.
		\end{equation}
		Thus $(\hat u,\hat v)\in L_{2s+1}\times L_{2s+1}$ satisfies \eqref{1D_problem}.
	\end{proof}	

	In order to find homogeneous solutions, we compute the fractional Laplacian of the homogeneous functions. Roughly speaking, the fractional operator decreases the degree of a homogeneous function by $2s$, while the sign of the resulting term plays a crucial role.
	\begin{lem}\label{lplc_s_homogns_fn}
		Let $2s<1$ and define $u_\gamma(t)=|t|^{\gamma-1}t$ with $2s<\gamma<2s+1$. Then $u_\gamma\in L_{2s+1}(\R)\cap C(\R)$ satisfies $(-\Delta)^su_\gamma=c_{s,\gamma}u_{\gamma-2s}$ in $\R_+$, where
		$$c_{s,\gamma}\begin{cases}
		>0,~2s<\gamma<1;\\
		=0,~\gamma=1;\\
		<0,~1<\gamma<1+2s.
		\end{cases}$$
	\end{lem}
	\begin{proof}
		We first consider the case $\gamma=1$. By definition \eqref{def_in_L2s+1},
		\begin{align*}
		(-\Delta)^su_1(t)=&c_s\int_0^\infty\left(\frac{1}{|t-r|^{1+2s}}-\frac{1}{|t+r|^{1+2s}}\right)(t-r)\D r+2c_st\int_0^\infty\frac{\D r}{|t+r|^{1+2s}}\\
		=&c_s\lim_{R\to\infty}\left(\int_0^R\left(\frac{1}{|t-r|^{1+2s}}-\frac{1}{|t+r|^{1+2s}}\right)(t-r)\D r+2t\int_0^R\frac{\D r}{|t+r|^{1+2s}}\right)\\
		=&c_s\lim_{R\to\infty}\left(\int_0^R\frac{(t-r)\D r}{|t-r|^{1+2s}}-\int_{-R}^0\frac{(t+r)\D r}{|t-r|^{1+2s}}+\int_{-R}^0\frac{2t\D r}{|t-r|^{1+2s}}\right)\\
		=&c_s\lim_{R\to\infty}\int_{-R}^R\frac{(t-r)\D r}{|t-r|^{1+2s}}=c_s\lim_{R\to\infty}\int_{R-t}^{t+R}\frac{\D r}{r^{2s}}=0.
		\end{align*}
		For general $\gamma$, we make the change of variables $r=tx$ in the following:
		\begin{align*}
		(-\Delta)^su_\gamma(t)=&c_{1,s}\int_0^\infty\left(\frac{1}{|t-r|^{1+2s}}-\frac{1}{|t+r|^{1+2s}}\right)(t^\gamma-r^\gamma)\D r+2c_{1,s}t^\gamma\int_0^\infty\frac{\D r}{|t+r|^{1+2s}}\\
		=&c_{1,s}t^{\gamma-2s}\int_0^\infty\left(\frac{1}{|1-x|^{1+2s}}-\frac{1}{|1+x|^{1+2s}}\right)(1-x^\gamma)\D x+2c_{1,s}t^{\gamma-2s}\int_0^\infty\frac{\D x}{(1+x)^{1+2s}}\\
		=&c_{1,s}\big(F(s,\gamma)+s^{-1}\big)t^{\gamma-2s},
		\end{align*}
		where
		\begin{align*}
		F(s,\gamma)=&\int_0^\infty\left(\frac{1}{|1-x|^{1+2s}}-\frac{1}{|1+x|^{1+2s}}\right)(1-x^\gamma)\D x\\
		=&\int_0^1\left(\frac{1}{(1-x)^{1+2s}}-\frac{1}{(1+x)^{1+2s}}\right)(1-x^\gamma)\D x\\
		&+\int_0^1\left(\frac{x^{1+2s}}{(1-x)^{1+2s}}-\frac{x^{1+2s}}{(1+x)^{1+2s}}\right)(1-x^{-\gamma})\frac{\D x}{x^2}\\
		=&\int_0^1\left(\frac{1}{(1-x)^{1+2s}}-\frac{1}{(1+x)^{1+2s}}\right)(1+x^{2s-1}-x^\gamma-x^{2s-\gamma-1})\D x.
		\end{align*}
		$F(s,\gamma)$ is decreasing in $\gamma$ because $1+x^{2s-1}-x^\gamma-x^{2s-\gamma-1}$ is decreasing in $\gamma$ for $x\in(0,1),\,\gamma\in(2s,2s+1)$. Since $F(s,1)+s^{-1}=0$, we obtain $F(s,\gamma)+s^{-1}>0$ for $\gamma\in(2s,1)$ and $F(s,\gamma)+s^{-1}<0$ for $\gamma\in(1,1+2s)$.
	\end{proof}
	
	Recall $p>q>0$ and $(q+2s)p+2s<1$. Let $\alpha=2s\cdot\dfrac{p+1}{1-pq},\,\beta=2s\cdot\dfrac{q+1}{1-pq}$, which satisfy
	$$\begin{cases}
	\alpha-2s=p\beta;\\
	\beta-2s=q\alpha.
	\end{cases}$$
	These agree with the scaling exponents of the Lane--Emden system. Since $\alpha,\beta\in(2s,1)$ and $0<pq<1$, Lemma~\ref{lplc_s_homogns_fn} allows us to choose $A,B>0$ such that 
	$$Ac_{s,\alpha}=B^p,~Bc_{s,\beta}=A^q,$$
	so that $A|t|^{\alpha-1}t,~B|t|^{\beta-1}t$ is a nontrivial pair of solutions of \eqref{1D_problem}.
	\begin{rmk}
		The scalar equation $(-\Delta)^sw=w^p$ also admits a homogeneous solution in the form $w(t)=C|t|^{\tfrac{2s}{1-p}-1}t,\,2s+p<1$.
	\end{rmk}
	
	The nonhomogeneous solution constructed above is asymptotically linear at infinity in the first component, growing more rapidly than the homogeneous profile. Indeed, we have the precise comparison:
	\begin{prop}\label{homogeneous_min}
		Let $u,v$ be any nontrivial solution to \eqref{problem}. Then $(u,v)\geq(u_*,v_*)$, i.e.
		$$u(t)\geq At^\alpha,~v(t)\geq Bt^\beta,~\forall t>0.$$
	\end{prop}
	\begin{proof}
		By the proof of Proposition~\ref{D=0_no_sol} in Section~\ref{sec_non_D=0}, we already obtain the estimate for $t\geq1$
		\begin{equation}\label{uv_est}
			u(t)\geq c_*t^{\tfrac{2s(p+1)}{1-pq}},~v(t)\geq c_*t^{\tfrac{2s(q+1)}{1-pq}}.
		\end{equation}
		By \eqref{kernel_at_infinity}, we derive
		$$u(t)\geq c_s'\int_0^t\big(|t-r|^{2s-1}-|t+r|^{2s-1}\big)v(r)^p\D r\geq c^*t^{2s-2}\int_0^trv(r)^p\D r.$$
		Similarly
		$$v(t)\geq c^*t^{2s-2}\int_0^tru(r)^q\D r.$$
		Lemma~\ref{lower_bound} implies
		$$u(t)\geq ct,~v(t)\geq ct,~0<t<1.$$
		Denoting $\tilde\alpha_0=\tilde\beta_0=0,~\tilde A_0=\tilde B_0=c$, an iteration argument again yields
		$$u(t)\geq\tilde A_kt^{\tilde\alpha_k},~v(t)\geq\tilde B_kt^{\tilde\beta_k},~0<t<1,$$
		where we define inductively
		$$\tilde\alpha_k:=2s+p\tilde\beta_{k-1},~\beta_k:=2s+q\tilde\alpha_{k-1},~\tilde A_k:=c^*\tilde B_{k-1}^p,~\tilde B_k:=c^*\tilde A_{k-1}^q.$$
		Passing to a limit, we find that \eqref{uv_est} holds for all $t>0$. Therefore, there existes $\theta\in(0,1)$ such that $(u,v)\geq(\theta u_*,\theta v_*)$. Then Lemma~\ref{lower_bound} implies
		\begin{align*}
			u(t)\geq&c_s'\int_0^\infty\big(|t-r|^{2s-1}-|t+r|^{2s-1}\big)v(r)^p\D r\\
			\geq&\theta^pc_s'\int_0^\infty\big(|t-r|^{2s-1}-|t+r|^{2s-1}\big)v_*(r)^p\D r\\
			=&\theta^pu_*(t);\\
			v(t)\geq&c_s'\int_0^\infty\big(|t-r|^{2s-1}-|t+r|^{2s-1}\big)u(r)^q\D r\\
			\geq&\theta^{pq}c_s'\int_0^\infty\big(|t-r|^{2s-1}-|t+r|^{2s-1}\big)u_*(r)^q\D r\\
			=&\theta^{pq}v_*(t),~\forall t>0.
		\end{align*}
		Similarly $u(t)\geq \theta^{pq}u_*(t)$. Repeating the estimate, we derive
		$$(u,v)\geq(\theta^{(pq)^k}u_*,\theta^{(pq)^k}v_*),~\forall k\in\mathbb{N}.$$
		Sending $k\to\infty$ yields the desired assertion.
	\end{proof}

	\begin{rmk}
		A full classification for the general regime $q<1-2s<p$ remains open. The precise critical exponent separating existence from non-existence is still unknown. However, based on our result, we conjecture that no solutions exist when
		$$0<q<1-2s<p\leq\frac{n+2s}{n-2s}~~\text{and}~~(q+2s)p+2s>1.$$
	\end{rmk}
{\bf{Acknowledgements.}} 

Li and Wang are partially supported by National Natural Science Foundation of China (NSFC Grant No. W2531006) and the Institute of Modern Analysis-A Frontier Research Center of Shanghai.
\medskip

{\bf{Data availability statement:}} Data will be made available on reasonable request.
\medskip

{\bf{Conflict of interest statement:}} There is no conflict of interest.

{\bf{Declaration on the Use of AI:}} During the preparation of this manuscript, ChatGPT was used to assist with symbolic calculations and the wording of the text. The authors were inspired by ChatGPT when proving Proposition~\ref{D=0_no_sol} and the universal lower bound Proposition~\ref{homogeneous_min}. All remaining results, including the Liouville-type simplification, Propsition~\ref{depend_n} and monotoicity, Proposition~\ref{increasing}, as well as the construction of two classes of solutions, were proved independently by the authors. The authors are responsible for the mathematical content and refrences.
	
\bibliographystyle{plain}
\bibliography{cite}

@book {Bertoin1996,
	AUTHOR = {Bertoin, Jean},
	TITLE = {L\'evy processes},
	SERIES = {Cambridge Tracts in Mathematics},
	VOLUME = {121},
	PUBLISHER = {Cambridge University Press, Cambridge},
	YEAR = {1996},
	PAGES = {x+265},
	ISBN = {0-521-56243-0},
	MRCLASS = {60J30 (60J55)},
	MRNUMBER = {1406564},
	MRREVIEWER = {Patrick\ Fitzsimmons},
}

@article{MetzlerKlafter2000,
	title = {The random walk's guide to anomalous diffusion: a fractional dynamics approach},
	journal = {Physics Reports},
	volume = {339},
	number = {1},
	pages = {1-77},
	year = {2000},
	issn = {0370-1573},
	doi = {https://doi.org/10.1016/S0370-1573(00)00070-3},
	url = {https://www.sciencedirect.com/science/article/pii/S0370157300000703},
	author = {Ralf Metzler and Joseph Klafter},
}

@article{JespersenMetzlerFogedby1999,
	title = {L\'evy flights in external force fields: Langevin and fractional Fokker-Planck equations and their solutions},
	author = {Jespersen, Sune and Metzler, Ralf and Fogedby, Hans C.},
	journal = {Phys. Rev. E},
	volume = {59},
	issue = {3},
	pages = {2736--2745},
	numpages = {0},
	year = {1999},
	month = {Mar},
	publisher = {American Physical Society},
	doi = {10.1103/PhysRevE.59.2736},
	url = {https://link.aps.org/doi/10.1103/PhysRevE.59.2736}
}

@article {ConstantinCordobaWu2001,
	AUTHOR = {Constantin, Peter and Cordoba, Diego and Wu, Jiahong},
	TITLE = {On the critical dissipative quasi-geostrophic equation},
	NOTE = {Dedicated to Professors Ciprian Foias and Roger Temam
	(Bloomington, IN, 2000)},
	JOURNAL = {Indiana Univ. Math. J.},
	FJOURNAL = {Indiana University Mathematics Journal},
	VOLUME = {50},
	YEAR = {2001},
	PAGES = {97--107},
	ISSN = {0022-2518,1943-5258},
	MRCLASS = {35Q35 (76B03 86A05)},
	MRNUMBER = {1855665},
	MRREVIEWER = {Emmanuel\ Grenier},
	DOI = {10.1512/iumj.2008.57.3629},
	URL = {https://doi.org/10.1512/iumj.2008.57.3629},
}

@article {KiselevNazarovVolberg2007,
	AUTHOR = {Kiselev, A. and Nazarov, F. and Volberg, A.},
	TITLE = {Global well-posedness for the critical 2{D} dissipative
	quasi-geostrophic equation},
	JOURNAL = {Invent. Math.},
	FJOURNAL = {Inventiones Mathematicae},
	VOLUME = {167},
	YEAR = {2007},
	NUMBER = {3},
	PAGES = {445--453},
	ISSN = {0020-9910,1432-1297},
	MRCLASS = {35Q35 (35B30 35K50)},
	MRNUMBER = {2276260},
	MRREVIEWER = {Xiaoming\ Wang},
	DOI = {10.1007/s00222-006-0020-3},
	URL = {https://doi.org/10.1007/s00222-006-0020-3},
}

@article {GrahamZworski2003,
	AUTHOR = {Graham, C. Robin and Zworski, Maciej},
	TITLE = {Scattering matrix in conformal geometry},
	JOURNAL = {Invent. Math.},
	FJOURNAL = {Inventiones Mathematicae},
	VOLUME = {152},
	YEAR = {2003},
	NUMBER = {1},
	PAGES = {89--118},
	ISSN = {0020-9910,1432-1297},
	MRCLASS = {58J50},
	MRNUMBER = {1965361},
	MRREVIEWER = {Andrew\ W.\ Hassell},
	DOI = {10.1007/s00222-002-0268-1},
	URL = {https://doi.org/10.1007/s00222-002-0268-1},
}

@article {ChangGonzalez2011,
	AUTHOR = {Chang, Sun-Yung Alice and Gonz\'alez, Mar\'ia del Mar},
	TITLE = {Fractional {L}aplacian in conformal geometry},
	JOURNAL = {Adv. Math.},
	FJOURNAL = {Advances in Mathematics},
	VOLUME = {226},
	YEAR = {2011},
	NUMBER = {2},
	PAGES = {1410--1432},
	ISSN = {0001-8708,1090-2082},
	MRCLASS = {58J60 (35R11)},
	MRNUMBER = {2737789},
	MRREVIEWER = {Luca\ Lorenzi},
	DOI = {10.1016/j.aim.2010.07.016},
	URL = {https://doi.org/10.1016/j.aim.2010.07.016},
}

@article {GonzalezQing2013,
	AUTHOR = {Gonz\'alez, Mar\'ia del Mar and Qing, Jie},
	TITLE = {Fractional conformal {L}aplacians and fractional {Y}amabe
	problems},
	JOURNAL = {Anal. PDE},
	FJOURNAL = {Analysis \& PDE},
	VOLUME = {6},
	YEAR = {2013},
	NUMBER = {7},
	PAGES = {1535--1576},
	ISSN = {2157-5045,1948-206X},
	MRCLASS = {35R11 (35J70 53A30)},
	MRNUMBER = {3148060},
	MRREVIEWER = {Pablo\ Ra\'ul\ Stinga},
	DOI = {10.2140/apde.2013.6.1535},
	URL = {https://doi.org/10.2140/apde.2013.6.1535},
}

@article {Lieb1983,
	AUTHOR = {Lieb, Elliott H.},
	TITLE = {Sharp constants in the {H}ardy-{L}ittlewood-{S}obolev and
	related inequalities},
	JOURNAL = {Ann. of Math. (2)},
	FJOURNAL = {Annals of Mathematics. Second Series},
	VOLUME = {118},
	YEAR = {1983},
	NUMBER = {2},
	PAGES = {349--374},
	ISSN = {0003-486X,1939-8980},
	MRCLASS = {42B25 (42A05)},
	MRNUMBER = {717827},
	MRREVIEWER = {Autorreferat},
	DOI = {10.2307/2007032},
	URL = {https://doi.org/10.2307/2007032},
}

@article {Gidas_Spruck,
	AUTHOR = {Gidas, B. and Spruck, J.},
	TITLE = {Global and local behavior of positive solutions of nonlinear
	elliptic equations},
	JOURNAL = {Comm. Pure Appl. Math.},
	FJOURNAL = {Communications on Pure and Applied Mathematics},
	VOLUME = {34},
	YEAR = {1981},
	NUMBER = {4},
	PAGES = {525--598},
	ISSN = {0010-3640,1097-0312},
	MRCLASS = {35J60 (53C20 58G05 81E10)},
	MRNUMBER = {615628},
	DOI = {10.1002/cpa.3160340406},
	URL = {https://doi.org/10.1002/cpa.3160340406},
}

@article {C_V2010,
	AUTHOR = {Caffarelli, Luis A. and Vasseur, Alexis},
	TITLE = {Drift diffusion equations with fractional diffusion and the
	quasi-geostrophic equation},
	JOURNAL = {Ann. of Math. (2)},
	FJOURNAL = {Annals of Mathematics. Second Series},
	VOLUME = {171},
	YEAR = {2010},
	NUMBER = {3},
	PAGES = {1903--1930},
	ISSN = {0003-486X,1939-8980},
	MRCLASS = {35Q35 (35B45 35B65 35R11 76B03 86A10)},
	MRNUMBER = {2680400},
	MRREVIEWER = {Denis\ Serre},
	DOI = {10.4007/annals.2010.171.1903},
	URL = {https://doi.org/10.4007/annals.2010.171.1903},
}

@article{Li2010,
	author = {Li, Dong and Zhuo, Ran},
	journal = {Proceedings of the American Mathematical Society},
	number = {8},
	pages = {2779--2791},
	publisher = {American Mathematical Society},
	title = {AN INTEGRAL EQUATION ON HALF SPACE},
	urldate = {2026-09-17},
	volume = {138},
	year = {2010}
}

@article {Zhuo2022,
	AUTHOR = {Zhuo, Ran and Li, Congming},
	TITLE = {Classification of anti-symmetric solutions to nonlinear
	fractional {L}aplace equations},
	JOURNAL = {Calc. Var. Partial Differential Equations},
	FJOURNAL = {Calculus of Variations and Partial Differential Equations},
	VOLUME = {61},
	YEAR = {2022},
	NUMBER = {1},
	PAGES = {Paper No. 17, 23},
	ISSN = {0944-2669,1432-0835},
	MRCLASS = {35R11},
	MRNUMBER = {4350191},
	DOI = {10.1007/s00526-021-02128-w},
	URL = {https://doi.org/10.1007/s00526-021-02128-w},
}

@article {ZhuoLi2023,
	AUTHOR = {Li, Congming and Zhuo, Ran},
	TITLE = {Classification of anti-symmetric solutions to the fractional
	{L}ane-{E}mden system},
	JOURNAL = {Sci. China Math.},
	FJOURNAL = {Science China. Mathematics},
	VOLUME = {66},
	YEAR = {2023},
	NUMBER = {4},
	PAGES = {723--744},
	ISSN = {1674-7283,1869-1862},
	MRCLASS = {35R11 (35B50 35B53)},
	MRNUMBER = {4577475},
	DOI = {10.1007/s11425-021-1952-1},
	URL = {https://doi.org/10.1007/s11425-021-1952-1},
}

@article{Chen2017,
	AUTHOR = {Chen, Wenxiong and Li, Congming and Li, Yan},
	TITLE = {A direct method of moving planes for the fractional
	{L}aplacian},
	JOURNAL = {Adv. Math.},
	FJOURNAL = {Advances in Mathematics},
	VOLUME = {308},
	YEAR = {2017},
	PAGES = {404--437},
	ISSN = {0001-8708,1090-2082},
	MRCLASS = {35R11 (35A01 35B09 35B50)},
	MRNUMBER = {3600062},
	MRREVIEWER = {\L ukasz\ P\l ociniczak},
	DOI = {10.1016/j.aim.2016.11.038},
	URL = {https://doi.org/10.1016/j.aim.2016.11.038},
}

@article{Li2021,
	AUTHOR = {Li, Congming and Wu, Leyun},
	TITLE = {Pointwise regularity for fractional equations},
	JOURNAL = {J. Differential Equations},
	FJOURNAL = {Journal of Differential Equations},
	VOLUME = {302},
	YEAR = {2021},
	PAGES = {1--36},
	ISSN = {0022-0396,1090-2732},
	MRCLASS = {35B65 (35R11)},
	MRNUMBER = {4310545},
	MRREVIEWER = {Nataliya\ Volodymyrivna\ Vasylyeva},
	DOI = {10.1016/j.jde.2021.08.027},
	URL = {https://doi.org/10.1016/j.jde.2021.08.027},
}

@article{Caffarelli1989,
	AUTHOR = {Caffarelli, Luis A. and Gidas, Basilis and Spruck, Joel},
	TITLE = {Asymptotic symmetry and local behavior of semilinear elliptic
	equations with critical {S}obolev growth},
	JOURNAL = {Comm. Pure Appl. Math.},
	FJOURNAL = {Communications on Pure and Applied Mathematics},
	VOLUME = {42},
	YEAR = {1989},
	NUMBER = {3},
	PAGES = {271--297},
	ISSN = {0010-3640,1097-0312},
	MRCLASS = {35J60 (35B40 58G30)},
	MRNUMBER = {982351},
	MRREVIEWER = {Robert\ McOwen},
	DOI = {10.1002/cpa.3160420304},
	URL = {https://doi.org/10.1002/cpa.3160420304},
}

@article{Chen2015,
	AUTHOR = {Chen, Wenxiong and Fang, Yanqin and Yang, Ray},
	TITLE = {Liouville theorems involving the fractional {L}aplacian on a
	half space},
	JOURNAL = {Adv. Math.},
	FJOURNAL = {Advances in Mathematics},
	VOLUME = {274},
	YEAR = {2015},
	PAGES = {167--198},
	ISSN = {0001-8708,1090-2082},
	MRCLASS = {35R11 (35A01 35A08 35B53)},
	MRNUMBER = {3318148},
	MRREVIEWER = {Kai\ Diethelm},
	DOI = {10.1016/j.aim.2014.12.013},
	URL = {https://doi.org/10.1016/j.aim.2014.12.013},
}

@article{Gidas_Ni_Nirenberg,
	AUTHOR = {Gidas, B. and Ni, Wei Ming and Nirenberg, L.},
	TITLE = {Symmetry of positive solutions of nonlinear elliptic equations
	in {${\bf R}\sp{n}$}},
	BOOKTITLE = {Mathematical analysis and applications, {P}art {A}},
	SERIES = {Adv. Math. Suppl. Stud.},
	VOLUME = {7a},
	PAGES = {369--402},
	PUBLISHER = {Academic Press, New York-London},
	YEAR = {1981},
	ISBN = {0-12-512801-0},
	MRCLASS = {35J60 (53C05 58G20)},
	MRNUMBER = {634248},
	MRREVIEWER = {D.\ E.\ Edmunds},
}

@article{Quaas2015,
	AUTHOR = {Quaas, Alexander and Xia, Aliang},
	TITLE = {Liouville type theorems for nonlinear elliptic equations and
	systems involving fractional {L}aplacian in the half space},
	JOURNAL = {Calc. Var. Partial Differential Equations},
	FJOURNAL = {Calculus of Variations and Partial Differential Equations},
	VOLUME = {52},
	YEAR = {2015},
	NUMBER = {3-4},
	PAGES = {641--659},
	ISSN = {0944-2669,1432-0835},
	MRCLASS = {35R11 (35B06 35B09 35B53)},
	MRNUMBER = {3311908},
	DOI = {10.1007/s00526-014-0727-8},
	URL = {https://doi.org/10.1007/s00526-014-0727-8},
}

@article{serrin1971,
	AUTHOR = {Serrin, James},
	TITLE = {A symmetry problem in potential theory},
	JOURNAL = {Arch. Rational Mech. Anal.},
	FJOURNAL = {Archive for Rational Mechanics and Analysis},
	VOLUME = {43},
	YEAR = {1971},
	PAGES = {304--318},
	ISSN = {0003-9527},
	MRCLASS = {31B20},
	MRNUMBER = {333220},
	MRREVIEWER = {Gottfried\ Anger},
	DOI = {10.1007/BF00250468},
	URL = {https://doi.org/10.1007/BF00250468},
}

@article{gidas1979,
	AUTHOR = {Gidas, B. and Ni, Wei Ming and Nirenberg, L.},
	TITLE = {Symmetry and related properties via the maximum principle},
	JOURNAL = {Comm. Math. Phys.},
	FJOURNAL = {Communications in Mathematical Physics},
	VOLUME = {68},
	YEAR = {1979},
	NUMBER = {3},
	PAGES = {209--243},
	ISSN = {0010-3616,1432-0916},
	MRCLASS = {35J25 (35B50)},
	MRNUMBER = {544879},
	MRREVIEWER = {\`E.\ M.\ Saak},
	URL = {http://projecteuclid.org/euclid.cmp/1103905359},
}

@article{troy1981,
	AUTHOR = {Troy, William C.},
	TITLE = {Symmetry properties in systems of semilinear elliptic
	equations},
	JOURNAL = {J. Differential Equations},
	FJOURNAL = {Journal of Differential Equations},
	VOLUME = {42},
	YEAR = {1981},
	NUMBER = {3},
	PAGES = {400--413},
	ISSN = {0022-0396,1090-2732},
	MRCLASS = {35J45 (80A30 92A09)},
	MRNUMBER = {639230},
	MRREVIEWER = {Richard\ J.\ Field},
	DOI = {10.1016/0022-0396(81)90113-3},
	URL = {https://doi.org/10.1016/0022-0396(81)90113-3},
}

@article {AoChanGonzalezWei2020,
	AUTHOR = {Ao, Weiwei and Chan, Hardy and Gonz\'alez, Mar\'ia del Mar and
	Wei, Juncheng},
	TITLE = {Bound state solutions for the supercritical fractional
	{S}chr\"odinger equation},
	JOURNAL = {Nonlinear Anal.},
	FJOURNAL = {Nonlinear Analysis. Theory, Methods \& Applications. An
	International Multidisciplinary Journal},
	VOLUME = {193},
	YEAR = {2020},
	PAGES = {111448, 21},
	ISSN = {0362-546X,1873-5215},
	MRCLASS = {35R11 (35J61)},
	MRNUMBER = {4062970},
	DOI = {10.1016/j.na.2019.02.002},
	URL = {https://doi.org/10.1016/j.na.2019.02.002},
}

@article{chen1991,
	AUTHOR = {Chen, Wenxiong and Li, Congming},
	TITLE = {Classification of solutions of some nonlinear elliptic
	equations},
	JOURNAL = {Duke Math. J.},
	FJOURNAL = {Duke Mathematical Journal},
	VOLUME = {63},
	YEAR = {1991},
	NUMBER = {3},
	PAGES = {615--622},
	ISSN = {0012-7094,1547-7398},
	MRCLASS = {35B05 (35J60)},
	MRNUMBER = {1121147},
	DOI = {10.1215/S0012-7094-91-06325-8},
	URL = {https://doi.org/10.1215/S0012-7094-91-06325-8},
}

@article{ChenLiOu,
	AUTHOR = {Chen, Wenxiong and Li, Congming and Ou, Biao},
	TITLE = {Classification of solutions for an integral equation},
	JOURNAL = {Comm. Pure Appl. Math.},
	FJOURNAL = {Communications on Pure and Applied Mathematics},
	VOLUME = {59},
	YEAR = {2006},
	NUMBER = {3},
	PAGES = {330--343},
	ISSN = {0010-3640,1097-0312},
	MRCLASS = {45E05 (35J60)},
	MRNUMBER = {2200258},
	MRREVIEWER = {Dorina\ I.\ Mitrea},
	DOI = {10.1002/cpa.20116},
	URL = {https://doi.org/10.1002/cpa.20116},
}

@article {Figueiredo1994ALT,
	AUTHOR = {de Figueiredo, D. G. and Felmer, P. L.},
	TITLE = {A {L}iouville-type theorem for elliptic systems},
	JOURNAL = {Ann. Scuola Norm. Sup. Pisa Cl. Sci. (4)},
	FJOURNAL = {Annali della Scuola Normale Superiore di Pisa. Classe di
	Scienze. Serie IV},
	VOLUME = {21},
	YEAR = {1994},
	NUMBER = {3},
	PAGES = {387--397},
	ISSN = {0391-173X,2036-2145},
	MRCLASS = {35B05 (35J45 35J60)},
	MRNUMBER = {1310633},
	MRREVIEWER = {Yoshihito\ Tomita},
	URL = {http://www.numdam.org/item?id=ASNSP_1994_4_21_3_387_0},
}

@article {SerrinZou1996,
	AUTHOR = {Serrin, James and Zou, Henghui},
	TITLE = {Non-existence of positive solutions of {L}ane-{E}mden systems},
	JOURNAL = {Differential Integral Equations},
	FJOURNAL = {Differential and Integral Equations. An International Journal
	for Theory \& Applications},
	VOLUME = {9},
	YEAR = {1996},
	NUMBER = {4},
	PAGES = {635--653},
	ISSN = {0893-4983},
	MRCLASS = {35J60 (35B05)},
	MRNUMBER = {1401429},
	MRREVIEWER = {C.\ A.\ Swanson},
}

@article {SerrinZou1998,
	AUTHOR = {Serrin, James and Zou, Henghui},
	TITLE = {Existence of positive solutions of the {L}ane-{E}mden system},
	JOURNAL = {Atti Sem. Mat. Fis. Univ. Modena},
	FJOURNAL = {Atti del Seminario Matematico e Fisico dell'Universit\`a{} di
	Modena},
	VOLUME = {46},
	YEAR = {1998},
	PAGES = {369--380},
	ISSN = {0041-8986},
	MRCLASS = {35J60 (35B05 35J45)},
	MRNUMBER = {1645728},
	MRREVIEWER = {Giovanni\ Leoni},
}

@article {Polacik2007,
	AUTHOR = {Pol\'a\v cik, Peter and Quittner, Pavol and Souplet, Philippe},
	TITLE = {Singularity and decay estimates in superlinear problems via
	{L}iouville-type theorems. {I}. {E}lliptic equations and
	systems},
	JOURNAL = {Duke Math. J.},
	FJOURNAL = {Duke Mathematical Journal},
	VOLUME = {139},
	YEAR = {2007},
	NUMBER = {3},
	PAGES = {555--579},
	ISSN = {0012-7094,1547-7398},
	MRCLASS = {35J60 (35B45 35J45 35J70)},
	MRNUMBER = {2350853},
	MRREVIEWER = {Alkis\ Tersenov},
	DOI = {10.1215/S0012-7094-07-13935-8},
	URL = {https://doi.org/10.1215/S0012-7094-07-13935-8},
}

@article {Souplet2009,
	AUTHOR = {Souplet, Philippe},
	TITLE = {The proof of the {L}ane-{E}mden conjecture in four space
	dimensions},
	JOURNAL = {Adv. Math.},
	FJOURNAL = {Advances in Mathematics},
	VOLUME = {221},
	YEAR = {2009},
	NUMBER = {5},
	PAGES = {1409--1427},
	ISSN = {0001-8708,1090-2082},
	MRCLASS = {35J47 (35A01 35B33 35B45 35J91)},
	MRNUMBER = {2522424},
	MRREVIEWER = {Scott\ N.\ Armstrong},
	DOI = {10.1016/j.aim.2009.02.014},
	URL = {https://doi.org/10.1016/j.aim.2009.02.014},
}

@article {MR3538413,
	AUTHOR = {Quaas, Alexander and Xia, Aliang},
	TITLE = {A {L}iouville type theorem for {L}ane-{E}mden systems
	involving the fractional {L}aplacian},
	JOURNAL = {Nonlinearity},
	FJOURNAL = {Nonlinearity},
	VOLUME = {29},
	YEAR = {2016},
	NUMBER = {8},
	PAGES = {2279--2297},
	ISSN = {0951-7715,1361-6544},
	MRCLASS = {35R11 (35B08 35B33 35B53 35J47 35S05 45K05)},
	MRNUMBER = {3538413},
	DOI = {10.1088/0951-7715/29/8/2279},
	URL = {https://doi.org/10.1088/0951-7715/29/8/2279},
}

@article {LuoZou16,
	AUTHOR = {Luo, Senping and Zou, Wenming},
	TITLE = {Liouville theorems for integral systems related to fractional
	{L}ane-{E}mden systems in {$\Bbb{R}^N_+$}},
	JOURNAL = {Differential Integral Equations},
	FJOURNAL = {Differential and Integral Equations. An International Journal
	for Theory \& Applications},
	VOLUME = {29},
	YEAR = {2016},
	NUMBER = {11-12},
	PAGES = {1107--1138},
	ISSN = {0893-4983},
	MRCLASS = {45G15 (35B53 35R11 65R20)},
	MRNUMBER = {3557314},
	MRREVIEWER = {Mario\ M.\ Coclite},
	URL = {http://projecteuclid.org/euclid.die/1476369332},
}

@article{OpenAI2026,
	AUTHOR = {OpenAI},
	TITLE = {The Subcritical H\'enon--Lane--Emden Conjecture},
	JOURNAL = {preprint},
	YEAR = {2026},
	URL = {https://github.com/openai/math/blob/main/preprints/The-Subcritical-Henon-Lane-Emden-Conjecture-September-24-2026/paper.pdf},
}

@misc{MeiWei2026,
	title={A proof of the fractional Lane--Emden conjecture}, 
	author={Linfeng Mei and Juncheng Wei},
	year={2026},
	eprint={2610.10188},
	archivePrefix={arXiv},
	primaryClass={math.AP},
	url={https://arxiv.org/abs/2610.10188}, 
}
 \end{document}